\documentclass[preprint,1p]{elsarticle}
\usepackage[pagewise]{lineno}
\usepackage{mathrsfs}
\usepackage{latexsym,amsfonts,amssymb,amsmath,amsthm,color,bbm}
\usepackage{graphicx,subfigure}
\usepackage{amsbsy,amsfonts,amsmath,amssymb,bbm,calc,caption,color,dsfont,graphics}
\usepackage{graphicx,ifthen,mathptmx,mathrsfs}
\usepackage{algorithm}
\usepackage{amsmath}
\usepackage{amsfonts}
\usepackage{epsfig}
\usepackage{epstopdf}
\usepackage[colorlinks=true, linkcolor=blue]{hyperref}
\usepackage{geometry}                  
\begin{document}
	\setlength{\baselineskip}{16pt}

	\newtheorem{theorem}{Theorem}[section]
	\newtheorem{lemma}{Lemma}[section]
	\newtheorem{proposition}{Proposition}[section]
	\newtheorem{definition}{Definition}[section]
	\newtheorem{example}{Example}[section]
	\newtheorem{corollary}{Corollary}[section]
	\newtheorem{assumption}{Assumption}[section]
	\newtheorem{remark}{Remark}[section]
	\numberwithin{equation}{section}
	\renewcommand{\labelenumi}{(\arabic{enumi})}

	\def\disp{\displaystyle}
	\def\undertex#1{$\underline{\hbox{#1}}$}
	\def\card{\mathop{\hbox{card}}}
	\def\sgn{\mathop{\hbox{sgn}}}
	\def\exp{\mathop{\hbox{exp}}}
	\def\OFP{(\Omega,{\cal F},\PP)}
	\newcommand\JM{Mierczy\'nski}
	\newcommand\RR{\ensuremath{\mathbb{R}}}
	\newcommand\EE{\ensuremath{\mathbb{E}}}
	\newcommand\CC{\ensuremath{\mathbb{C}}}
	\newcommand\QQ{\ensuremath{\mathbb{Q}}}
	\newcommand\ZZ{\ensuremath{\mathbb{Z}}}
	\newcommand\NN{\ensuremath{\mathbb{N}}}
	\newcommand\PP{\ensuremath{\mathbb{P}}}
	\newcommand\abs[1]{\ensuremath{\lvert#1\rvert}}
	\newcommand\normf[1]{\ensuremath{\lVert#1\rVert_{f}}}
	\newcommand\normfRb[1]{\ensuremath{\lVert#1\rVert_{f,R_b}}}
	\newcommand\normfRbone[1]{\ensuremath{\lVert#1\rVert_{f, R_{b_1}}}}
	\newcommand\normfRbtwo[1]{\ensuremath{\lVert#1\rVert_{f,R_{b_2}}}}
	\newcommand\normtwo[1]{\ensuremath{\lVert#1\rVert_{2}}}
	\newcommand\norminfty[1]{\ensuremath{\lVert#1\rVert_{\infty}}}
	
	\newcommand\del[1]{}

\begin{frontmatter}

\title{Finite-dimensional approximations of random attractor for stochastic discrete complex Ginzburg-Landau equations }
\tnotetext[Supported]{The research is supported by National Natural Science Foundation of China (12371198).}

\author{Xinjie Fang}
\author{Jianhua Huang}
\author{Fang Su\corref{cor1}}
\author{Jun Ouyang}
\cortext[cor1]{Corresponding author, sufang@nudt.edu.cn}
\address{College of Sciences, National University of Defense Technology, Changsha Hunan, 410073, P.R.China}

\begin{abstract}
In this paper, we apply an implicit Euler scheme to discretize the complex Ginzburg-Landau equation and prove the existence of a numerical attractor for the discrete Ginzburg-Landau system. We establish the upper semicontinuity of the numerical attractor with respect to the global attractor as the time step tends to zero. Furthermore, we provide finite-dimensional approximations for three types of attractors (global, numerical, and random), and demonstrate the existence of truncated attractors along with their convergence as the dimension of the state space tends to infinity. Finally, we prove the existence of a random attractor and establish the upper semi-continuity both of the global random attractor and the truncated random attractor.
\end{abstract}

\begin{keyword}
Implicit Euler scheme \sep Numerical attractor \sep Random attractor \sep Hausdorff distance.\\
\end{keyword}

\end{frontmatter}

\noindent\textbf{Mathematics Subject Classification} 34D45 $\cdot$ 37K60 $\cdot$  65L20 

\section{Introduction}\label{intro}

The Ginzburg-Landau (G-L) equation is a fundamental model for describing superconducting phenomena. By introducing a complex order parameter, it characterizes the macroscopic quantum behavior and phase transitions of superconductors below the transition temperature. Proposed by Ginzburg and Landau in 1950 based on Landau's second-order phase transition theory, this theory not only successfully explained the critical phenomena of superconductors and the distribution of internal magnetic fields but was also further developed by Abrikosov to predict type-II superconductors and their vortex lattice structures \cite{Guo-2018}. Subsequently, in 1959, Gorkov demonstrated from a microscopic perspective that the Ginzburg-Landau equation is a macroscopic manifestation of the microscopic theory of superconductivity under specific conditions. For their foundational contributions, Ginzburg and Abrikosov were awarded the Nobel Prize in Physics in 2003.

The Ginzburg-Landau equation reveals numerous fascinating physical phenomena, including magnetic flux quantization, vortex dynamics, superconducting domain wall structures, and various topological defects \cite{Kittel-2022}. These phenomena have established the equation as a centerpiece not only in the study of traditional superconductors but also in a broad range of physical and engineering fields, such as high-temperature superconductivity, Bose-Einstein condensation, and nonlinear optics.

Extensive studies have been conducted on the theoretical analysis and numerical simulation of the Ginzburg-Landau equation. For example, the monograph by Guo Boling et al. systematically summarizes the global solutions, asymptotic behavior, and connections with harmonic mappings of this equation \cite{Guo-2018}; Kong Youchao et al. obtained multiple exact traveling wave solutions in the form of envelope waves using the homogeneous balance method and F-expansion method \cite{Kong-2016}. However, with the emergence of new quantum materials and complex physical systems (e.g., anisotropic superconductors), re-evaluating the classical Ginzburg-Landau equation within the frameworks of fluctuations, random perturbations, and fractional calculus has become crucial for more accurately describing physical phenomena that are susceptible to external environments or internal disorder.

In the context of the stochastic Ginzburg-Landau equation, several recent works have begun to explore the existence of solutions for fractional models with random perturbations or investigate the approximation of invariant measures \cite{Rey-2014}. While these advances provide key insights into understanding the statistical properties of the system, a systematic analysis of the long-term dynamical behavior of the system, particularly from the perspective of attractor theory, remains a significant challenge.

Attractors offer a powerful mathematical framework for characterizing the long-term dynamical behavior of solutions to partial differential equations. In recent years, substantial efforts have been dedicated to studying attractors in deterministic and stochastic systems, such as reaction-diffusion systems, the Gross-Pitaevskii equation, quantum vortices, and the Navier-Stokes equations. Recently, Mantzavinos et al. investigated the well-posedness of the complex Ginzburg-Landau equation on finite intervals and the suppression of chaotic behavior via finite-dimensional boundary feedback control \cite{Alkin-2024}.

To the best of our knowledge, although deterministic and stochastic Ginzburg-Landau equations play fundamental roles in describing the long-term dynamical behavior of superconducting and superfluid systems, research on the existence, structural characteristics, finite-dimensional approximations, and random attractors of their global attractors remains largely open, especially in complex scenarios involving fractional derivatives, strong anisotropy, or practical boundary control. This paper aims to systematically review the current research status of the Ginzburg-Landau equation, with a particular focus on its dynamical behavior and the development of attractor theory, thereby providing references and directions for future research.

In this paper, we focus on the following non-autonomous stochastic complex Ginzburg-Landau equation:
\begin{equation}\label{1.1}
\frac{d u(t)}{d t}=-(\lambda+i\mu)\frac{\partial^2u(t)}{\partial x^2}-(\gamma+i\beta)u-(k+i\nu)|u|^{p}u+g(t)+a u \circ \frac{dW}{d t},
\end{equation}

Later, we will attempt to approximate the global attractor $\mathcal{A}$ using numerical attractors and random attractors, respectively. This approach aims to approximate the global attractor from both microscopic and macroscopic perspectives.

Our first objective is to study the existence and finite-dimensional approximation of numerical attractors for the complex Ginzburg-Landau equation, where time discretization is performed using an implicit Euler scheme (IES). For the deterministic complex Ginzburg-Landau equation, we first prove the existence and uniqueness of solutions for the implicit Euler scheme and demonstrate that the discrete-time system possesses a numerical attractor $\mathcal{A}_{\epsilon}$. As shown in Theorem \ref{dingli3.3}, when the time step $\epsilon \to 0$, this attractor upper semi-converges to the global attractor $\mathcal{A}$ of the continuous-time lattice model, moreover, we prove its upper semi-continuity. Next, we consider the finite dimensional approximation of the numerical attractor $\mathcal{A}_{m}^{\varepsilon}$ and prove its convergence to $\mathcal{A}^{\varepsilon}$ as $m\to\infty$, which is shown in Theorem \ref{dingli4.2}.

The second goal is to study the stochastic complex Ginzburg-Landau equation, prove the existence of a random attractor using the Ornstein Uhlenbeck process and tail estimates for the random solution. We then investigate the convergence property between the random attractor $\mathcal{A}^{a}(\omega)$ and the global attractor $\mathcal{A}$ as $a\to0$, which is diplayed in Theorem \ref{dingli5.1}. Besides, we prove that the truncated random attractor $\mathcal{A}_{m}^{a}(\omega)$ converges to the random attractor $\mathcal{A}^{a}(\omega)$ as $m\to\infty$, which is shown in Theorem \ref{dingli5.2}. Furthermore, we give the upper semi-continuity in Theorem \ref{suijishangbanlianxu} Based on the convergence of solutions between the deterministic and random equations in the truncation sense, we demonstrate the convergence property from $\mathcal{A}_{m}^{a}(\omega)$ to the truncated global attractor $\mathcal{A}^{m}$ as the noise intensity approaches zero, which will be presented in Theorem \ref{dingli5.3}. Finally we give the upper semi-continuity of the truncated random attractor in Theorem \ref{jieduanshangbanlianxu}.

\section{Well-posedness and discretization error}\label{sec:2}

We first apply the implicit Euler scheme to obtain the lattice system for the deterministic complex Ginzburg-Landau equation:
\begin{equation}\label{gl1}
\left\{\begin{array}{cl}
&\frac{d u_{j}(t)}{d t}=(\lambda+i\mu)(-u_{j-1}+2u_{j}-u_{j+1})-(\gamma+i\beta)u_{j}-(k+i\nu)|u_{j}|^{p}u_{j}+g_{j}(t),\\
&u_{j}(0)=u_{0,j},j \in \mathbb{Z}.
\end{array}\right.
\end{equation}
According to \cite{Gu-2016}, let $\ell^2$ be the Hilbert space of complex-valued bi - infinite sequences that are square summable. The inner product on $\ell^2$ is defined as
\begin{equation*}
\left(u,v\right)=\sum_{j\in \mathbb{Z}}u_{j}\overline{v_{j}},\quad \forall u = (u_{j})_{j \in \ZZ}, v = (v_{j})_{j\in\ZZ} \in \ell^{2},
\end{equation*}
and the norm is ${\|u\|}^{2}_{2} = \left(u,u\right)$. Additionally, we have 
\begin{equation*}
\begin{aligned}
	\ell^2:=\left\{u=\left(u_i\right)_{i \in \mathbb{Z}}:\|u\|^2=\sum_{i \in \mathbb{Z}}\left|u_i\right|^2<\infty\right\},\\
	\ell^p:=\left\{u=\left(u_i\right)_{i \in \mathbb{Z}}:\|u\|^{p}_{p}=\sum_{i \in \mathbb{Z}}\left|u_i\right|^p<\infty\right\}.
\end{aligned}
\end{equation*}
For $\forall j\in\ZZ$ and $u = (u_{j})_{j\in\ZZ}$, we defined some operators as follows:
\begin{equation}\label{chafen}
\begin{aligned}
	&\Lambda : \ell^{2} \rightarrow \ell^{2}, \ (\Lambda u)_{j} = -u_{j-1} + 2u_{j} - u_{j+1},\\
	&D^{+} : \ell^{2} \rightarrow \ell^{2}, \ (D^{+}u)_{j} = u_{j+1} - u_{j},\\
	&D^{-} : \ell^{2} \rightarrow \ell^{2}, \ (D^{-}u)_{j} = u_{j-1} - u_{j},
\end{aligned}
\end{equation}
then we have $(\Lambda u,u) = \| D^{+}u\|^{2}$. According to \cite{Yang-2022}, all operators are bounded on $\ell^{2}$ with $\| \Lambda \| \leq 4 $ and $\|D^{+}\|=\| D^{-} \| \leq 2 $.

Therefore the equation \eqref{gl1} can be rewritten as 
\begin{equation}\label{gl}
\left\{\begin{array}{cl}
&\frac{d u(t)}{d t}=(\lambda+i\mu)\Lambda u-(\gamma+i\beta)u-(k+i\nu)|u|^{p}u+g := Fu,\\
&u(0)=u_0.
\end{array}\right.
\end{equation}
where $\lambda, \gamma,k,\alpha,p > 0$, $\left|u\right|^{p}u=\left(\left|u_{j}\right|^{p}u_{j}\right)_{j \in \mathbb{Z}}$, and $g=(g_{j})_{j \in \mathbb{Z}}$, the operator $F$ is called the vector field of the complex Ginzburg-Landau equation. In this article, we use the IES to obtain the discrete - time system of \eqref{gl} as follows:
\begin{equation}\label{glt}
	u_{n}^{\varepsilon} = u_{n-1}^{\varepsilon} + \varepsilon Fu_{n}^{\varepsilon},
\end{equation}
where the $\varepsilon$ is the time step, $u_{n}^{\varepsilon} = (u_{n,j}^{\varepsilon})_{j\in\ZZ}$.

First, we introduce two useful Lipschitz constants for the operator $F$.
\begin{lemma}\label{yinli2.1}
	Let the operator $F$ be defined as in \eqref{gl}. Then for any $u, v\in B_{r}$, the following equalities hold:
	\begin{align}
	&\|Fu\| \leq (4\lambda+4|\mu|+\gamma+|\beta|)r + (k+|\nu|)r^{p+1} +\|g\| =: M_{r},\label{jie}\\
	&\|Fu - Fv\| \leq L_{r}\|u-v\|, \text{with\ } L_{r}:= 4\lambda+4|\mu|+\gamma+|\beta| + C_{p}r^{p}(k+|\nu|).\label{lianxu}
	\end{align}
	where
	\begin{equation*}
		B_{r}:=\left\{u\in\ell^{2};\|u\|\leq r\right\},\ \forall r \textgreater 0.
	\end{equation*}
\end{lemma}
\begin{proof}[Proof]
	For any $r\textgreater 0$, let $u\in B_{r}$, we obtain
	\begin{equation*}
	\begin{split}
	\|Fu\| &= \|(\lambda+i\mu)\Lambda u-(\gamma+i\beta)u-(k+i\nu)|u|^{p}u+g\|\\
		   &\leq | \lambda+i\mu | \| \Lambda u \| + | \gamma+i\beta | \| u \| + | k+i\nu | \| u \|^{p+1} + \| g \|\\
		   &\leq (4\lambda+4|\mu|+\gamma+|\beta|)r + (k+|\nu|)r^{p+1} +\|g\| =: M_{r}.
	\end{split}
	\end{equation*}
	Suppose that $u,v\in B_{r}$, then we get 
	\begin{equation}
		\| |u|^{p}u - |v|^{p}v\| \leq C_{p}r^{p}\| u-v \|.
	\end{equation}
	Where $C_{p}$ is a constant depending only on $p$ \cite{Liu-2024a}. In fact, over the complex field, $\forall z_{1},z_{2}\in\CC$ and $p \geq 0$, $\exists C_{0}(p)$, such that
	\begin{equation}\label{ss}
		||z_1|^{p} - |z_2|^{p}| \leq C_0(p) (|z_1|^p + |z_2|^p) |z_1 - z_2|.
	\end{equation}
	By \eqref{ss}, we have
	\begin{equation*}
	\begin{split}
		\||u|^{p}u - |v|^{p}v\|^2 &= \sum_{j \in \mathbb{Z}} \left||u_j|^p u_j- |v_j|^p v_j \right|^2\\
			&\leq C_0^2(p) \sum_{j \in \mathbb{Z}} \left(|u_j|^p + |v_j|^p\right)^2 |u_j - v_j|^2\\
			&\leq \sqrt{4C_0^2(p)} r^p \|u - v\| = 2C_0(p) r^p \|u - v\|,
	\end{split}
	\end{equation*}
	where we use $(|u_j|^p + |v_j|^p)^2 \leq (r^p + r^p)^2 = (2r^p)^2 = 4r^{2p}$, and $2C_{0}(p) = C_{p}$.
	
	Therefore, we can get
	\begin{equation*}
	\begin{split}
	\|Fu - Fv\| &= \| (\lambda+i\mu)(\Lambda u - \Lambda v) - (\gamma+i\beta)(u - v) - (k+i\nu)(|u|^{p}u - |v|^{p}v) \|\\
			    &\leq (4\lambda+4|\mu|+\gamma+|\beta| + C_{p}r^{p}(k+|\nu|))\|u-v\|\\
			    & = L_{r}\|u-v\|.
	\end{split}
	\end{equation*}
\end{proof}

\subsection{Existence of a unique solution to the IES of the complex Ginzburg-Landau equation}
To prove the unique existence of the IES for the complex Ginzburg-Landau equation, we choose the special radius and timesize as follows:
\begin{equation}\label{banjinheshijian}
	r^{*} = \sqrt{\eta + \frac{c_{1}}{\gamma - 4\lambda}},\quad \varepsilon^{*}=\min\left\{ \frac{1}{M_{r^{*}+1}},\ \frac{1}{1+L_{r^{*}+1}} \right\} , 
\end{equation}
where $\eta\geq0$ and $c_{1} = \frac{p+1}{p+2}\frac{1}{\left[k(p+2)\right]^{\frac{1}{p+1}}}\|g\|_{\frac{p+2}{p+1}}^{\frac{p+2}{p+1}}$ and we assume that $\gamma\textgreater4\lambda$.  
\begin{theorem}\label{dingli2.1}
	For each $\varepsilon \in \left( 0, \varepsilon^{*} \right] $, and $u_{0} \in B_{r^{*}}$, the equation \eqref{glt} has a unique solution $u_{n}^{\varepsilon}(u_{0}) \in B_{r^{*}}$.
\end{theorem}
\begin{proof}[Proof]
	\begin{description}
	\item[Step 1: ](positive invariance) We prove the solution $u_{n}^{\varepsilon}$ of the IES \eqref{glt} satisfies $u_{n}^{\varepsilon} \in B_{r^{*}}$, if $u_{n-1}^{\varepsilon} \in B_{r^{*}}$. Taking the inner product of \eqref{glt} by $u_{n}^{\varepsilon}$ to find
	\begin{equation*}
		\|u_n^\varepsilon\|^2 = (u_{n-1}^{\varepsilon}, u_n^\varepsilon) + \varepsilon (\lambda + i\mu ) (\Lambda u_n^\varepsilon, u_n^\varepsilon) - \varepsilon (\gamma + i\beta) \|u_n^\varepsilon\|^2\\
		-\varepsilon (k + i\nu) (|u_n^\varepsilon|^p u_n^\varepsilon, u_n^\varepsilon) + \varepsilon (g, u_n^\varepsilon),
	\end{equation*}
	then we take the real part
	\begin{equation*}
		\begin{aligned}
		\|u_{n}^\varepsilon\|^2 &= \operatorname{Re}[(u_{n-1}^{\varepsilon}, u_{n}^\varepsilon)] + \varepsilon \lambda (\Lambda u_{n}^\varepsilon, u_{n}^\varepsilon) - \varepsilon \gamma \|u_{n}^\varepsilon\|^2 - \varepsilon k (|u_{n}^\varepsilon|^p u_{n}^\varepsilon, u_{n}^\varepsilon) + \varepsilon \operatorname{Re}[(g, u_{n}^\varepsilon)]\\
		&\leq \frac{1}{2} \|u_{n-1}^{\varepsilon}\|^2 + \frac{1}{2} \|u_{n}^\varepsilon\|^2 + 4\varepsilon \lambda \|u_{n}^\varepsilon\|^2 - \varepsilon \gamma \|u_{n}^\varepsilon\|^2 - \varepsilon k \|u_{n}^\varepsilon\|^{p+2}_{p+2} + \varepsilon \operatorname{Re}[(g, u_{n}^\varepsilon)],
		\end{aligned}
	\end{equation*}
	by Holder's inequality and Young's inequality
	\begin{equation*}
	\begin{split}
		\operatorname{Re}[(g, u_{n}^\varepsilon)] &\leq | (g,u_{n}^{\varepsilon}) | \leq \|u_{n}^{\varepsilon}\|_{p+2} \|g\|_{\frac{p+2}{p+1}}\\
		&\leq k\|u_{n}^\varepsilon\|_{p+2}^{p+2}+\frac{p+1}{p+2}\frac{1}{\left[k(p+2)\right]^{\frac{1}{p+1}}}\|g\|_{\frac{p+2}{p+1}}^{\frac{p+2}{p+1}},
	\end{split}
	\end{equation*}
	thus, we obtain
	\begin{equation}\label{ss1}
	\begin{aligned}
		\|u_{n}^\varepsilon\|^2&\leq\frac{1}{2}\|u_{n-1}^{\varepsilon}\|^2+\frac{1}{2}\|u_{n}^\varepsilon\|^2+(4\varepsilon\lambda-\varepsilon\gamma)\|u_{n}^\varepsilon\|^2-\varepsilon k\|u_{n}^\varepsilon\|_{p+2}^{p+2}\\
		&\quad+\varepsilon k\|u_{n}^\varepsilon\|_{p+2}^{p+2}+\varepsilon\frac{p+1}{p+2}\frac{1}{\left[k(p+2)\right]^{\frac{1}{p+1}}}\|g\|_{\frac{p+2}{p+1}}^{\frac{p+2}{p+1}}\\
		&=:\frac{1}{2}\|u_{n-1}^{\varepsilon}\|^2+\frac{1}{2}\|u_{n}^\varepsilon\|^2+(4\varepsilon\lambda-\varepsilon\gamma)\|u_{n}^\varepsilon\|^2+\varepsilon c_1,
	\end{aligned}
	\end{equation}
	then we can reorganize \eqref{ss1}
	\begin{equation}
	\begin{split}
		\|u_{n}^\varepsilon\|^2&\leq\frac{1}{1+(2\gamma-8\lambda)\varepsilon}(\|u_{n-1}^{\varepsilon}\|^2+2\varepsilon c_1)\\
		&\leq \frac{\eta}{1+(2\gamma-8\lambda)\varepsilon} + \frac{c_{1}}{\gamma - 4\lambda}\\
		&\leq \eta + \frac{c_{1}}{\gamma - 4\lambda} = r^{*},
	\end{split}
	\end{equation}
	which implies $u_{n}^{\varepsilon} \in B_{r^{*}}$.
	\item[Step 2:](unique existence for n=1) For $u_{0} \in B_{r^{*}}$ and $\varepsilon \in \left(0,\varepsilon^{*}\right]$, we define the operator $\Phi_{u_{0}}^{\varepsilon}$ by
	\begin{equation*}
	\Phi_{u_{0}}^{\varepsilon}y = u_{0} + \varepsilon Fy, \ \forall y\in \ell^{2}.
	\end{equation*}
	If $y\in B_{r^{*}+1}$, we deduce from \eqref{jie} and \eqref{banjinheshijian} that
	\begin{equation*}
		\| \Phi_{u_{0}}^{\varepsilon}y \| \leq \|u_{0}\| + \varepsilon\|Fy\|\leq r^{*} + \varepsilon^{*}M_{r^{*}+1}\leq r^{*}+1.
	\end{equation*}
	Thus, the operator $\Phi_{u_0}^{\varepsilon}: B_{r*+1} \to B_{r*+1}$ is well defined. From \eqref{lianxu} and \eqref{banjinheshijian}, for any $y,z\in \ell^{2}$, we have
	\begin{equation}
		\|\Phi ^\varepsilon u_0(y) - \Phi ^\varepsilon u_0(z)\| = \varepsilon\| Fy-Fz \|\leq\frac{L^{r^{*}+1}}{1+L^{r^{*}+1}}\| y-z \|. 
	\end{equation}
	where $\frac{L^{r^{*}+1}}{1+L^{r^{*}+1}} \textless 1$. Therefore, $\forall \varepsilon \in \left(0, \varepsilon^*\right]$ and $u_0 \in B_{r^*}$, the operator $\Phi ^\varepsilon: B_{r^*+1} \to B_{r^*+1}$ is a contraction. 
	By the contraction mapping principle, $\Phi_{u_0} ^\varepsilon$ has a unique fixed point $u^\varepsilon \in B_{r^*+1}$. This is a solution to \eqref{glt}, and the solution is unique in $B_{r^*+1}$.
	\item [Step 3:](unique existence for all $n \in \NN$) We follow a similar proof proceess based on Step 2. Using $u_{n-1}^\varepsilon\in B_{r^*}$ as initial data, we obtain that for all $n\in\mathbb{N}$ and $\varepsilon\in(0,\varepsilon^*]$, equation \eqref{glt} has a unique solution $u_n^\varepsilon\in B_{r^*}$. This completes the proof.
	\end{description}
\end{proof}
\subsection{Taylor expansion and discretization error}
According to lemma \ref{yinli2.1}, the equation \eqref{gl} has a local unique solution $u = (u_{i}(t))_{i\in\ZZ}$, for $t\in \left[0,T_{\max}\right)$ as established in \cite{Kl-1977}. We now demonstrate that this local solution extends globally, i.e., $T_{\max}=\infty$.
\begin{lemma}\label{yinli2.2}
	Assume that $4\lambda-\gamma \textless 0$, for any $u_{0} \in \ell^{2}$, the lattice system \eqref{gl} of the complex Ginzburg-Landau equation has a unique solution $u(\cdot,u_{0}) \in C(\left[0,\infty\right))$. Furthermore, the ball $B_{r^{*}}$ in Theorem \ref{dingli2.1} is positively invariant and absorbing for the lattice system \eqref{gl}, i.e.
	\begin{align}
		u(t,u_{0})\in B_{r^{*}},\ \forall t \geq 0, \, u_{0}\in B_{r^{*}},\\
		\limsup_{t \rightarrow +\infty} \sup_{\|u_{0}\|\leq r} \| u(t,u_{0}) \|\textless r^{*},\ \forall r \textgreater 0.
	\end{align}
\end{lemma}
\begin{proof}[Proof]
	First we take the real part of the inner product of the lattice system \eqref{gl} with $u(t)$, by Holder's inequality and Young's inequality, we have
	\begin{equation*}
		\begin{aligned}
		\frac{d\|u\|^2}{dt}&=2\lambda(\Lambda u,u)-2\gamma\|u\|^2-2k(|u|^{p}u,u)+2\operatorname{Re}(g,u)\\
		&=2\lambda (D^{+}u, D^{+}u) -2\gamma \|u\|^2 - 2k\sum_{j \in \mathbb{Z}} |u_j|^p u_j \overline{u_j} + 2\operatorname{Re}(g, u)\\
		&=2 \lambda ||D^{+}u||^2 - 2 \gamma \|u\|^2 - 2k\sum_{j \in \mathbb{Z}} |u_j|^{p+2} + 2\operatorname{Re}(g, u)\\
		&\leq 8\lambda\|u\|^2-2\gamma\|u\|^2-2k\|u\|^{p+2}_{p+2}+2|(g,u)|\\
		&\leq 8\lambda\|u\|^2-2\gamma\|u\|^2-2k\|u\|^{p+2}_{p+2}+2k\|u\|^{p+2}_{p+2}+2\frac{p+1}{p+2}\frac{1}{\left[k(p+2)\right]^{\frac{1}{p+1}}} \|g\|^{\frac{p+2}{p+1}}_{\frac{p+2}{p+1}}\\
		&\leq (8\lambda-2\gamma)\|u\|^2+2\frac{p+1}{p+2}\frac{1}{\left[k(p+2)\right]^{\frac{1}{p+1}}} \|g\|^{\frac{p+2}{p+1}}_{\frac{p+2}{p+1}}\\
		&=(8\lambda-2\gamma)\|u\|^2+2c_1.
		\end{aligned}
	\end{equation*}
	By Gronwall Lemma, we deduce that
	\begin{equation}\label{quanjuyoujie}
		\| u(t) \|^{2} \leq \| u_{0} \|^{2}{\rm{e}}^{(8\lambda-2\gamma)t}+\frac{2c_{1}}{8\lambda-2\gamma}\left({\rm{e}}^{(8\lambda-2\gamma)t}-1\right).
	\end{equation}
	Thus, we have 
	\begin{equation*}
		\max_{0\leq t \leq T}\| u(t) \|^{2}\leq \| u_{0} \|^{2} + \frac{2c_{1}}{2\gamma-8\lambda} \textless +\infty, \ \forall T\textgreater0.
	\end{equation*}
	Therefore, a global solution exists, and its uniqueness is guaranteed by \eqref{lianxu}.
	
	For any $u_{0}\leq r^{*}$ and $t\geq0$, where $r^{*}$ is difined in \eqref{banjinheshijian}, it can be deduced from \eqref{quanjuyoujie} that
	\begin{equation*}
	\begin{aligned}
		\| u(t,u_{0}) \|^{2} &\leq \| u_{0} \|^{2}{\rm{e}}^{(8\lambda-2\gamma)t}+\frac{2c_{1}}{8\lambda-2\gamma}\left({\rm{e}}^{(8\lambda-2\gamma)t}-1\right)\\
		&\leq (\eta + \frac{c_{1}}{\gamma-4\lambda}){\rm{e}}^{(8\lambda-2\gamma)t}  +\frac{2c_{1}}{8\lambda-2\gamma}\left({\rm{e}}^{(8\lambda-2\gamma)t}-1\right)\\
		&\leq \eta + \frac{c_{1}}{\gamma-4\lambda} = (r^{*})^{2}.
	\end{aligned}
	\end{equation*}
	Therefore, $B_{r^{*}}$ is positively invariant.
	
	Fruthermore, if $\| u_{0} \|\leq r,\ \forall r\textgreater0$, it follows from \eqref{quanjuyoujie} that
	\begin{equation*}
		\| u(t,u_{0}) \|^{2} \leq r^{2}{\rm{e}}^{(8\lambda-2\gamma)t}  +\frac{2c_{1}}{8\lambda-2\gamma}\left({\rm{e}}^{(8\lambda-2\gamma)t}-1\right).
	\end{equation*}
	\begin{description}
	\item[(i)] If $r^{2}\leq \frac{c_{1}}{\gamma-4\lambda}$, it is obvious that
	\begin{equation*}
		r^{2}{\rm{e}}^{(8\lambda-2\gamma)t}  +\frac{2c_{1}}{8\lambda-2\gamma}\left({\rm{e}}^{(8\lambda-2\gamma)t}-1\right)\leq \eta + \frac{c_{1}}{\gamma-4\lambda},\ \forall t\textgreater0.
	\end{equation*}
	\item[(ii)] If $r^{2}\textgreater \frac{c_{1}}{\gamma-4\lambda}$, 
	\begin{equation*}
		r^{2}{\rm{e}}^{(8\lambda-2\gamma)t}  +\frac{2c_{1}}{8\lambda-2\gamma}\left({\rm{e}}^{(8\lambda-2\gamma)t}-1\right)\leq \eta + \frac{c_{1}}{\gamma-4\lambda},\ \forall t \geq \frac{\ln \eta-\ln(r^{2}+\frac{c_{1}}{4\lambda-\gamma})}{8\lambda-2\gamma}.
	\end{equation*}
	\end{description}
	Thus, $B_{r^{*}}$ is absorbing, and we complete the proof.
\end{proof}

To account the discretization error, we establish a one-step second-order Taylor expansion for the continuous-time solutions.
\begin{lemma}\label{yinli2.3}
	For $t_n-t_{n-1}=\varepsilon > 0$ and $n \in \mathbb{N}$, there exists an operator $\mathcal{O}_{n,\varepsilon}:B_{r^*} \to \ell^2$ such that
	\begin{equation}\label{taylor}
		u(t_{n-1},y) = \varepsilon Fu(t_{n},y) + \varepsilon^{2}\mathcal{O}_{n,\varepsilon}y,\ \forall y \in B_{r^{*}},
	\end{equation}
	where $F$ is the vector field in \eqref{gl}. Moreover the operator$\mathcal{O}_{n,\varepsilon}$ is uniformly bounded on $B_{r^{*}}$, i.e.
	\begin{equation}
		\| \mathcal{O}_{n,\varepsilon}y \| \leq C_{0}, \ \forall n\in\NN, y\in B_{r^{*}}
	\end{equation}
\end{lemma}
\begin{proof}[Proof]
	According to \eqref{gl}
	\begin{equation}\label{p}
	u(t_{n-1}, y) - u(t_n, y) = -\int_{t_{n-1}}^{t_n} F(u(s, y)) ds,
	\end{equation}
	\begin{equation*}
		\begin{aligned}
		\mathcal{O}_{n,\varepsilon}(y) &= -\frac{1}{\varepsilon^2} \int_{t_{n-1}}^{t_n} F(u(s, y)) ds + \frac{1}{\varepsilon} F(u(s, y)) ds\\
		&= \frac{1}{\varepsilon^2} \int_{t_{n-1}}^{t_n} [F(u(t_n, y)) - F(u(s, y))] ds.
	\end{aligned}
	\end{equation*}
	By Lemma \ref{yinli2.3}, for $\forall s\in [ t_{n-1},t_{n} ]$, we have
	\begin{equation}\label{pp}
		\|F(u(t_n, y)) - F(u(s, y))\| \leq L_{r^{*}} \|u(t_n) - u(s)\|,
	\end{equation}
	where $L_{r^{*}}$ is defined in \eqref{lianxu}. By \eqref{p}, we have
	\begin{equation*}
	\begin{aligned}
		\|u(t_n) - u(s)\| &= \left\| -\int_{t_n}^{s} F(u(\tau, y)) d\tau \right\|\\
		&\leq \int_{s}^{t_n} \|F(u(\tau, y))\| d\tau\\
		&\leq \varepsilon M_{r^{*}}.
	\end{aligned}
	\end{equation*}
	Thus, we obtain
	\begin{equation*}
	\begin{aligned}
		\|\mathcal{O}_{n,\varepsilon}(y)\| &\leq \frac{1}{\varepsilon^2} \int_{t_{n-1}}^{t_n} \|F(u(t_n, y)) - F(u(s, y))\| ds \\
		&\leq \frac{L_{r^{*}}}{\varepsilon^2} \int_{t_{n-1}}^{t_n} \|u(t_n) - u(s)\| ds \\
		&= \frac{1}{2}L_{r^{*}} M_{r^{*}} =: C_0.
	\end{aligned}
	\end{equation*}
\end{proof}
In the following therorem, we provide the error estimates between the continuous-time solution and its discrete-time solution, demonstrating that the discrete-time solution convergers to the continuous-time solution as $\varepsilon \to 0^{+}$.
\begin{theorem}\label{dingli2.2}
	Denoting by $u(t,y)$ and $u_{n}^{ \varepsilon}(y)$ the solution to \eqref{gl} and \eqref{glt} with the initial condition $y \in B_{r^{*}}$. Then we get
	\begin{equation}
		\| u(\varepsilon,u_{n-1}^{\varepsilon}(y)) - u_{n}^{\varepsilon}(y) \| \leq L_{r^{*}}M_{r^{*}}L_{r^{*}+1}\varepsilon^{2}, \ \forall \varepsilon \in (0,\varepsilon^{*}], \ n\in \NN.
	\end{equation}
	Furthermore, for every $T\textgreater0$, we get
	\begin{equation}
		\|u(t_n, y) - u_{t_n}^\varepsilon(y)\| \leq \frac{M_r}{2} e^{L_r T} \varepsilon, \quad \forall t_n := \varepsilon n \in (0, T].
	\end{equation}
\end{theorem}
The proofs of Therorem \ref{dingli2.2} can be found in \cite{Liu-2024b}, where the $F$ is the vector field  of the complex Ginzburg-Landau equation.

\section{Upper semi-convergence of the numerical attractor}
\subsection{Existence of a unique global attractor for the continuous-time system}
We consider the existence of a compact global attractor $\mathcal{A}$ of the lattice system \eqref{gl}:
\begin{equation*}
	u(t,\mathcal{A}) = \mathcal{A},\ \forall t\geq0\ and\  \lim_{t \to +\infty}d_{\ell^{2}}(u(t,B_{r}),\mathcal{A}),\ \forall r \textgreater0.
\end{equation*}
\begin{theorem}\label{dingli3.1}
	The lattice system \eqref{gl} has a unique global attractor, which can be expressed as:
	\begin{equation*}
		\mathcal{A} = \bigcap_{T\textgreater0}\overline{\bigcup_{t\geq T}u(t,B_{r^{*}})},
	\end{equation*}
	where $r^{*}$ is defined in \eqref{banjinheshijian}.
\end{theorem}
\begin{proof}[Proof]
	According to Lemma \ref{yinli2.2}, the set $B_{r^{*}}$ is absorbing for the \eqref{gl}. Following a similar result as stated in \cite{Han-2020,Li-2023}, it is sufficient to prove the asymptotic tails property.
	
	We choose the cutoff function $\xi: \RR^{+}\to [0,1]$, which is continuously differentiable and satisfies the following properties:
	\begin{equation*}
		\xi(s) = 
		\begin{cases}
		0,\ 0\leq s \leq 1,\\
		1,\ s\geq2.
		\end{cases}
	\end{equation*}
	Then there exists a constant $c_{2}\textgreater0$ such that for every $l\in\NN$ and $j\in\ZZ$
	\begin{equation}
		|\xi_{l,j+1} - \xi_{l,j}| \leq \frac{c_{2}}{l}, \ where\  \xi_{l,j}=\xi(\frac{|j|}{l}),\ (\xi_{l,j})_{j\in\ZZ} =: \xi_{l}.
	\end{equation}\label{jieduan}
	By \eqref{chafen}, we get
	\begin{equation}
		\| D^{+}\xi_{l} \|_{\infty} \leq \frac{c_{2}}{l}, \ \forall l \in \NN
	\end{equation}
	Let $u = u(t,y)$ be the continuous-time solution of the lattice system \eqref{gl} for $y\in B_{r}$, where $r\textgreater0$ is any chosen radius.
	
	Noting that $(\Lambda u ,v) = (D^{+}u, D^{+}v), \ D^{+}(uv) = vD^{+}u + uD^{+}v$, by taking the inner product of equation \eqref{gl} with $\xi_{l}u$, then we take the real part:
	\begin{equation}\label{lianxuweibuguji}
	\begin{split}
		\frac{1}{2} \frac{d}{dt} \sum_{j\in\ZZ}\xi_{l,j}|u_{j}|^{2}(t) &= \operatorname{Re} \left[((\lambda+i\mu)\Lambda u-(\gamma+i\beta)u-(k+i\nu)|u|^{p}u+g,\ \xi_{l}u)\right],\\
		&=\operatorname{Re} \left[ (\lambda+i\mu)(D^{+}u, D^{+}(\xi_{l}u)) - (\gamma+i\beta)(u, \xi_{l}u) - (k+i\nu)(|u|^{p}u,\xi_{l}u) + (g,\xi_{l}u)      \right].
	\end{split}
	\end{equation}
	By Young's inequality and Holder's inequality
	\begin{equation}\label{diyixiang1}
	\begin{split}
		\operatorname{Re}\left[ \lambda(D^{+}u, D^{+}(\xi_{l}u))  \right] &=  \operatorname{Re} \lambda \left[ (D^{+}u,\xi_{l}D^{+}u) + (D^{+}u,uD^{+}\xi_{l}) \right]\\
		&\leq \lambda|\sum_{j\in\ZZ}\xi_{l,j}(|u_{j+1}|^{2} - u_{j+1}\overline{u_{j}} - \overline{u_{j+1}}u_{j} + |u_{j}|^{2})| + 
		\lambda\| D^{+}u \| \| u \| \| D^{+}\xi_{l} \|_{\infty}\\
		&\leq \lambda\sum_{j\in\ZZ}\xi_{l,j}(2| u_{j+1} |^{2} + 2| u_{j} |^{2}) + 2\lambda \| u \|^{2} \frac{c_{2}}{l},
	\end{split}
	\end{equation}
	where
	\begin{equation*}
	\begin{split}
		\sum_{j\in\ZZ}\xi_{l,j}|u_{j+1}|^{2} &= \sum_{j\in\ZZ} (\xi_{l,j} - \xi_{l,j+1})|u_{j+1}|^{2} + \sum_{j\in\ZZ}\xi_{l,j}|u_{j}|^{2}\\
		& = -\sum_{j\in\ZZ}(D^{+}\xi_{l})_{j}|u_{j+1}|^{2} +\sum_{j\in\ZZ}\xi_{l,j+1}|u_{j+1}|^{2}\\
		&\leq \frac{c_{2}}{l}\| u \|^{2} + \sum_{j\in\ZZ}\xi_{l,j}|u_{j}|^{2}.
	\end{split}
	\end{equation*}
	Then we can rewrite the \eqref{diyixiang1}
	\begin{equation}
		\operatorname{Re}\left[ \lambda(D^{+}u, D^{+}(\xi_{l}u))  \right] \leq 4\lambda\sum_{j\in\ZZ}\xi_{l,j}| u_{j} |^{2} + 4\lambda\frac{c_{2}}{l} r^{2}.
	\end{equation}
	Similar to \eqref{diyixiang1}, we have
	\begin{equation}\label{diyixiang2}
	\begin{split}
		&\quad\operatorname{Re}\left[  i\mu(D^{+}u, D^{+}(\xi_{l}u))   \right]\\
		 &= \operatorname{Re}\left[ i\mu \sum_{j\in\ZZ}\xi_{l,j}|D^{+}u_{j}|^{2} + i\mu(D^{+}u,uD^{+}\xi_{l})    \right]\\
		 &= \operatorname{Re}\left[ i\mu(D^{+}u,uD^{+}\xi_{l})  \right]
		 \leq \mu \| D^{+}u \| \| u \| \| D^{+}\xi_{l} \|_{\infty}
		 \leq 2\mu \frac{c_{2}}{l}\| u \|^{2}.
	\end{split}
	\end{equation}
	For the second term and the third term of \eqref{lianxuweibuguji}, we have
	\begin{equation}\label{dierxiang}
		\operatorname{Re}\left[ -(\gamma+i\beta)(u,\xi_{l}u) \right] = -\gamma\sum_{j\in\ZZ}\xi_{l,j}|u_{j}|^{2}.
	\end{equation}
	\begin{equation}\label{disanxiang}
		\operatorname{Re}\left[ -(k+i\nu)(|u|^{p}u,\xi_{l}u)  \right] = -k\sum_{j\in\ZZ}\xi_{l,j}|u_{j}|^{p+2}.
	\end{equation}
	According to Young's inequality, the last term is estimated as follow:
	\begin{equation}\label{disixiang}
	\begin{split}
		 \operatorname{Re}(g,\xi_{l}u)& \leq |(g,\xi_{l}u)|
		 = (|g|,|\xi_{l}u|)
		 = \sum_{j\in\ZZ}\xi_{l,j}|g_{j}||u_{j}|\\
		& \leq k\sum_{j\in\ZZ} \xi_{l,j}|u_{j}|^{p+2} + \frac{p+1}{p+2}\left[ \frac{1}{k(p+2)} \right]^{\frac{1}{p+1}} \sum_{j\in\ZZ} \xi_{l,j}|g_{j}|^{\frac{p+2}{p+1}}.
	\end{split}
	\end{equation}
	Substituting \eqref{diyixiang1}-\eqref{disixiang} into \eqref{lianxuweibuguji}, we have
	\begin{equation*}
		\frac{1}{2}\frac{d}{dt} \sum_{j\in\ZZ}\xi_{l,j}|u_{j}|^{2}(t) = (4\lambda+2\mu)\frac{c_{2}}{l}\| u \|^{2} + (4\lambda-\gamma)\sum_{j\in\ZZ}\xi_{l,j}|u_{j}|^{2} + \frac{p+1}{p+2}\left[ \frac{1}{k(p+2)} \right]^{\frac{1}{p+1}} \sum_{j\in\ZZ} \xi_{l,j}|g_{j}|^{\frac{p+2}{p+1}}.
	\end{equation*}
	By Gronwall's inequality, for any $u(0)\in B_{r}$ we conclude that
	\begin{equation}\label{lianxuweibugujijieguo}
		\sum_{j\in\ZZ}\xi_{l,j}|u_{j}|^{2}(t) \leq {\rm{e}}^{(8\lambda-2\gamma)t}r^{2} + \frac{1-{\rm{e}}^{(8\lambda-2\gamma)t}}{\gamma-4\lambda} \left[ (4\lambda+2\mu)\frac{c_{2}}{l}\|u\|^{2} +  \frac{p+1}{p+2}\left[ \frac{1}{k(p+2)} \right]^{\frac{1}{p+1}} \sum_{j\in\ZZ} \xi_{l,j}|g_{j}|^{\frac{p+2}{p+1}} \right].
	\end{equation}
	Since $\gamma \textgreater 4\lambda$, $B_{r^{*}}$ is attracting, $\| u\|^{2}$ is bounded as $t\to\infty$. Besides, because of $g\in\ell^{2}$, $\sum_{j\in\ZZ}\xi_{l,j}|g_{j}|^{\frac{p+2}{p+1}}$ is bounded. Thus we obtain
	\begin{equation*}
		\lim_{l,t\to \infty} \sup_{y\in B_{r}} \sum_{|j|\geq 2l} |u_{j}|^{2} \leq \lim_{l,t\to \infty} \sup_{y\in B_{r}} \sum_{j\in\ZZ}\xi_{l,j}|u_{j}(t)|^{2} =0.
	\end{equation*}
	The asymptotic tails property of the continuous-time solution is proven.
\end{proof}
\subsection{Existence of a unique numerical attractor for the IES}
The discrete-time solution $u^{\varepsilon}_{n}$ in the Theorem \ref{dingli2.1} defines a discrete dynamical system on the ball $B_{r^{*}}$, expressed as:
\begin{equation*}
	S_{\varepsilon}(n): B_{r^{*}}\to B_{r^{*}},\ S_{\varepsilon}(n)y = u_{n}^{\varepsilon}(y),\ y\in B_{r^{*}},\ n\in\NN,
\end{equation*}
for any $\varepsilon\in ( 0,\varepsilon^{*} ]$. Moreover, we can deduce that $S_{\varepsilon}(n+m) = S_{\varepsilon}(n)S_{\varepsilon}(m)$.

A compact set ${\mathcal{A}}_{\varepsilon}$ in $\ell^{2}$ is called a numerical attractor of the IES \eqref{glt} if ${\mathcal{A}}_{\varepsilon}$ satisfies the following two conditions:
\begin{description}
\item[1] The compact set ${\mathcal{A}}_{\varepsilon}$ is invariant, meaning that
\begin{equation*}
	S_{\varepsilon}(n){\mathcal{A}}_{\varepsilon} = {\mathcal{A}}_{\varepsilon},\ n\in\NN.
\end{equation*}
\item[2] The compact set ${\mathcal{A}}_{\varepsilon}$ is attracting, meaning that
\begin{equation*}
	\lim_{t \to \infty} d_{\ell^{2}} (S_{\varepsilon}(n), {\mathcal{A}}_{\varepsilon}) = 0,
\end{equation*}
\item where $d_{\ell^{2}}$ denotes the Hausdorff semi-distance, definedas $d(A,B) = \sup_{a\in A} \inf_{b\in B}\| a-b \|$.
\end{description}
\begin{theorem}\label{dingli3.2}
	For any $\varepsilon\in ( 0,\varepsilon^{*} ]$,  the IES \eqref{glt} for the Ginzburg-Landau equation has a unique numerical attractor, which can be expressed as follows:
	\begin{equation*}
		{\mathcal{A}}_{\varepsilon} = \bigcap_{m\in\NN} \overline{\bigcup_{n\geq m}S_{\varepsilon}(n)B_{r^{*}}}.
	\end{equation*}
\end{theorem}
\begin{proof}[Proof]
	Since $B_{r^{*}}$ is absorbing for the discrete dynamical system, and similarly to Theorem \ref{dingli3.1}, it is sufficient to prove the asymptotic tails property. That is, for any $M \textgreater 0$, there exist $N,I\in\NN$ such that the following inequality holds:
	\begin{equation*}
		\sum_{|j|\geq I} |u_{n,j}^{\varepsilon}|^{2} \leq M,\ \forall n\geq N,\ y\in B_{r^{*}},
	\end{equation*}
	$u^{\varepsilon}_{n}(y) = \left( u_{n,j}^{\varepsilon} \right)_{j\in\ZZ}$ is the unique solution of IES \eqref{glt}. For convenience in the proof, we use notation $u_{n}$ instead of $u^{\varepsilon}_{n}(y)$.
	
	Also we take the inner product of equation \eqref{glt} with $\xi_{l}u_{n}$, then take the real part:
	\begin{equation}\label{lisanweibuguji1}
		\begin{split}
		\sum_{j\in \mathbb{Z}}\xi_{l,j}|u_{n,j}|^2 &= \operatorname{Re}(u_{n-1},\xi_{l}u_{n}) + \varepsilon\operatorname{Re}\left(\left(\lambda + i\mu\right)\Lambda u_{n}, \xi_{l}u_{n}\right)\\
		&\quad- \varepsilon\gamma\sum_{j \in \mathbb{Z}}\xi_{l,j}\left|u_{n,j}\right|^{2} - \varepsilon k \sum_{j\in\ZZ}\xi_{l,j}|u_{n,j}|^{p+2} + \varepsilon\operatorname{Re}\left(g, \xi_{l}u_{n}\right)
		\end{split}
	\end{equation}
	The first term on the right-hand side of \eqref{lisanweibuguji1} is estimated as below:
	\begin{equation}\label{lisandiyixiang}
		\operatorname{Re}(u_{n-1,j},\xi_{l,j}u_{n,j}) \leq |(u_{n-1,j},\xi_{l,j}u_{n,j})|\leq \frac{1}{2}\sum_{j\in\ZZ}\xi_{l,j}u_{n-1,j} + \frac{1}{2}\sum_{j\in\ZZ}\xi_{l,j}u_{n,j}.
	\end{equation} 
	Using a similar method as in the derivation of \eqref{diyixiang1}-\eqref{disixiang}, we have
	\begin{align}
		&\varepsilon\operatorname{Re}\left(\lambda \Lambda u_{n}, \xi_{l}u_{n}\right) \leq 4\varepsilon\lambda \sum_{j\in\ZZ} \xi_{l,j}|u_{n,j}|^{2} + 4\varepsilon\lambda\frac{c_{2}}{l}\|u_{n}\|^{2},\label{lisandierxiang}\\ 
		&\varepsilon\operatorname{Re}\left(i\mu \Lambda u_{n}, \xi_{l}u_{n}\right) \leq 2\varepsilon\mu \frac{c_{2}}{l}\|u_{n}\|^{2},\label{lisandisanxiang}\\ 
		&\varepsilon\operatorname{Re}\left(g, \xi_{l}u_{n}\right) \leq \varepsilon k\sum_{j\in\ZZ} \xi_{l,j}|u_{n,j}|^{p+2} + \varepsilon\frac{p+1}{p+2}\left[ \frac{1}{k(p+2)} \right]^{\frac{1}{p+1}} \sum_{j\in\ZZ} \xi_{l,j}|g_{j}|^{\frac{p+2}{p+1}}\label{lisandisixiang}. 
	\end{align}
	Substituting \eqref{lisandiyixiang}-\eqref{lisandisixiang}, we obtain
	\begin{equation}\label{lisanweibuguji2}
		\begin{split}
			\sum_{j\in \mathbb{Z}}\xi_{l,j}|u_{n,j}|^2 & \leq \frac{1}{2}\sum_{j\in \mathbb{Z}}\xi_{l,j}|u_{n-1,j}|^2 + \left( \frac{1}{2}-\varepsilon(\gamma-4\lambda)  \right)\sum_{j\in \mathbb{Z}}\xi_{l,j}|u_{n,j}|^2 + \varepsilon\left[(4\lambda+2\mu)\frac{c_{2}}{l}\|u_{n}\|^{2}\right.\\
			&\left.\quad + \frac{p+1}{p+2}\left[ \frac{1}{k(p+2)} \right]^{\frac{1}{p+1}} \sum_{j\in\ZZ} \xi_{l,j}|g_{j}|^{\frac{p+2}{p+1}}\right]
		\end{split}
	\end{equation}
	When $l\to\infty$, the last term in \eqref{lisanweibuguji2} tends to zero, we can choose $l_{M} \textgreater 0$ such that
	\begin{equation*}
		(4\lambda+2\mu)\frac{c_{2}}{l}\|u_{n}\|^{2} + \frac{p+1}{p+2}\left[ \frac{1}{k(p+2)} \right]^{\frac{1}{p+1}} \sum_{j\in\ZZ} \xi_{l,j}|g_{j}|^{\frac{p+2}{p+1}} \textless \frac{M}{2}(\gamma-4\lambda), \ \forall l \geq l_{M}.
	\end{equation*}
	Thus \eqref{lisanweibuguji2} can be rewritten as:
	\begin{equation}\label{lisanweibuguji3}
		\sum_{j\in \mathbb{Z}}\xi_{l,j}|u_{n,j}|^2 \leq \frac{2\varepsilon(\gamma-4\lambda)}{1+2\varepsilon(\gamma-4\lambda)}\frac{M}{2} +\frac{1}{1+2\varepsilon(\gamma-4\lambda)}\sum_{j\in\ZZ}\xi_{l,j}|u_{n-1,j}|^{2}.
	\end{equation}
	As $n\to\infty$, for any $l\geq l_{M}$, we get
	\begin{equation*}
		\begin{split}
			\sum_{j\in \mathbb{Z}}\xi_{l,j}|u_{n,j}|^2 &\leq \frac{1}{\left[ 1+2\varepsilon(\gamma-4\lambda) \right]^{n}} \sum_{j\in\ZZ}\xi_{l,j}|u_{0,j}|^{2} + \frac{M}{2} \sum_{m=1}^{n}\frac{2\varepsilon(\gamma-4\lambda)}{\left[ 1+2\varepsilon(\gamma-4\lambda) \right]^{m}}\\
			& \textless \frac{(r^{*})^{2}}{\left[ 1+2\varepsilon(\gamma-4\lambda) \right]^{n}} + \frac{M}{2} \to \frac{M}{2}.
		\end{split}
	\end{equation*}
	Therefore we have
	\begin{equation*}
		\sum_{|j|\geq 2l}|u_{n,j}|^{2} \leq \sum_{j\in\ZZ}\xi_{l,j}|u_{n,j}|^{2} \textless M, \ \forall n\geq N, \ l\geq l_{M}.
	\end{equation*}
	Choosing $I = 2l_{M}$, the theorem is proved.
\end{proof}
\subsection{Approximation from numerical attractor to the global attractor}
In this section, we give a theorem regarding the upper semi-convergence of the numerical attractor to the global attractor, and the upper semi-continuity of the numerical attractor under the Hausdorff semi-distance.
	\begin{theorem}\label{dingli3.3}
		Let \(\mathcal{A}_\varepsilon (\varepsilon \in (0,\varepsilon^*])\) be the numerical attractor and \(\mathcal{A}\) the global attractor. Then:
	\begin{equation}\label{shangbanshoulian}
		\lim_{\varepsilon \to 0^+} d_{\ell^2}(\mathcal{A}_\varepsilon, \mathcal{A}) = 0
	\end{equation}
	Moreover, \(\mathcal{A}_\varepsilon\) is upper semicontinuous on \((0,\varepsilon^*]\):
	\begin{equation}\label{shangbanlianxu}
		\lim_{\varepsilon \to \varepsilon_0} d_{\ell^2}(\mathcal{A}_\varepsilon, \mathcal{A}_{\varepsilon_0}) = 0, \quad \forall \varepsilon_0 \in (0,\varepsilon^*]
	\end{equation}
	\end{theorem}
	\begin{proof}[Proof]
		The proof of this theorem generally follows a standard approach. As in \cite{Liu-2024a} and \cite{Liu-2024b}, the proof utilizes discretization errors to establish approximation properties. Therefore, using a contradiction argument and the discretization error provided in Theorem \ref{dingli2.2}, the result can be similarly proved.
	\end{proof}
 
\section{Finite dimensional approximation of numerical attractor}
\subsection{Existence of truncated numerical attractor} 
We consider the finite dimensional approximation of the IES \eqref{glt}. For $m\in\NN$ we consider the system of implicit difference equations in $\CC^{2m+1}$ by truncating the infinite dimensional system as follows:
\begin{equation}\label{jieduangl1}
	\begin{cases}
	\ u_{n,-m}^{\varepsilon,m} &= \quad u_{n-1,-m}^{\varepsilon,m} + \varepsilon\left[ (\lambda+i\mu)\left( -u_{n,m}^{\varepsilon,m} +2u_{n,-m}^{\varepsilon,m}-u_{n,-m+1}^{\varepsilon,m} \right) -(\gamma+i\beta)u_{n,-m}^{\varepsilon,m} \right.\\
	& \left. \qquad- (k+i\nu)|u_{n,-m}^{\varepsilon,m}|^{p}u_{n,-m}^{\varepsilon,m} + g_{-m} \right],\\

	\ u_{n,-m+1}^{\varepsilon,m} &= \quad u_{n-1,-m+1}^{\varepsilon,m} + \varepsilon\left[ (\lambda+i\mu)\left( -u_{n,-m}^{\varepsilon,m} +2u_{n,-m+1}^{\varepsilon,m}-u_{n,-m+2}^{\varepsilon,m} \right) -(\gamma+i\beta)u_{n,-m+1}^{\varepsilon,m} \right.\\
	& \left. \qquad- (k+i\nu)|u_{n,-m+1}^{\varepsilon,m}|^{p}u_{n,-m+1}^{\varepsilon,m} + g_{-m+1} \right],\\	
	
	&\qquad \qquad \qquad \vdots\\
	
	\ u_{n,m-1}^{\varepsilon,m} &= \quad u_{n-1,m-1}^{\varepsilon,m} + \varepsilon\left[ (\lambda+i\mu)\left( -u_{n,m-2}^{\varepsilon,m} +2u_{n,m-1}^{\varepsilon,m}-u_{n,m}^{\varepsilon,m} \right) -(\gamma+i\beta)u_{n,m-1}^{\varepsilon,m} \right.\\
	& \left. \qquad- (k+i\nu)|u_{n,m-1}^{\varepsilon,m}|^{p}u_{n,m-1}^{\varepsilon,m} + g_{m-1} \right],\\
	
	\ u_{n,m}^{\varepsilon,m} &= \quad u_{n-1,m}^{\varepsilon,m} + \varepsilon\left[ (\lambda+i\mu)\left( -u_{n,m-1}^{\varepsilon,m} +2u_{n,m}^{\varepsilon,m}-u_{n,-m}^{\varepsilon,m} \right) -(\gamma+i\beta)u_{n,m}^{\varepsilon,m} \right.\\
	& \left. \qquad- (k+i\nu)|u_{n,m}^{\varepsilon,m}|^{p}u_{n,m}^{\varepsilon,m} + g_{m} \right].\\
	\end{cases}
\end{equation}
The initial condition is given by
\begin{equation*}
	u_{0}^{\varepsilon,m} = z \in \CC^{2m+1}.
\end{equation*} 
In \eqref{jieduangl1}, we use the boundary condition as in \cite{Han-2020}, where $u_{n,m}^{\varepsilon,m} = u_{n,-m-1}^{\varepsilon,m},\ u_{n,m+1}^{\varepsilon,m} = u_{n,-m}^{\varepsilon,m}$.

According to \cite{Liu-2024a}, we have
\begin{equation}\label{lambdam}
	\Lambda_{m} = 
	\begin{pmatrix}
		 2 & -1 & 0 & \cdots & 0 & 0 & -1\\
		-1 &  2 & -1 & \cdots & 0 & 0 & 0\\
		\vdots & \vdots & \vdots & \ddots & \vdots & \vdots & \vdots\\
		0 & 0 & 0 & \cdots & -1 & 2 & -1\\
		-1 & 0 & 0 & \cdots & 0 & -1 & 2\\
	\end{pmatrix}
	.
\end{equation}
Similarly, we can obtain $D_{m}^{+}$, then we have $(\Lambda_{m}u_{n}^{\varepsilon,m},u_{n}^{\varepsilon,m})=(D_{m}^{+}u_{n}^{\varepsilon,m},D_{m}^{+}u_{n}^{\varepsilon,m})$. Let $g^{m} = (g_{j})_{|j|\leq m} \in \CC^{2m+1}$, $u_{n}^{\varepsilon,m} = (u_{n,j}^{\varepsilon,m})_{|j|\leq m} \in \CC^{2m+1}$, then we can rewrite the truncated system\eqref{jieduangl1} as:
\begin{equation}\label{jieduangl2}
\begin{cases}
	&u_{n}^{\varepsilon,m} = u_{n-1}^{\varepsilon,m} + \varepsilon\left( (\lambda+i\mu)\Lambda_{m}u_{n}^{\varepsilon,m} - (\gamma+i\beta)u_{n}^{\varepsilon,m} - (k+i\nu)|u_{n}^{\varepsilon,m}|^{p}u_{n}^{\varepsilon,m} + g^{m} \right),\\
	&u_{0}^{\varepsilon,m} = z\in \CC^{2m+1}.
\end{cases}
\end{equation}
The truncated vector field is also expressed as:
\begin{equation*}
	F_{m}z = (\lambda+i\mu)\Lambda_{m}z - (\gamma+i\beta)z -(k+i\nu)|z|^{p}z + g^{m}, \ \forall z\in\CC^{2m+1}.
\end{equation*}
Let $B_{r^{*}}^{m}$ be the ball of radius $r\textgreater0$ in $\CC^{2m+1}$. According to Lemma \ref{yinli2.1}, for any $y,z \in B_{r^{*}}$, we have
\begin{align}
\begin{split}\label{jieduanyoujie}
	\| F_{m}y \| &\leq (4\lambda+4|\mu|+\gamma+|\beta|)r + (k+|\nu|)r^{p+1} +\|g^{m}\|  \\
	& \leq (4\lambda+4|\mu|+\gamma+|\beta|)r + (k+|\nu|)r^{p+1} +\|g\| =M_{r}, \\
\end{split}\\
\begin{split}\label{jieduanlianxu}
		 \| F_{m}y-F_{m}z \| \leq L_{r} \| y-z \|,
\end{split}
\end{align}
where $M_{r}$ and $L_{r}$ are difined in Lemma \ref{yinli2.1}. By using \eqref{jieduanyoujie} and \eqref{jieduanlianxu}, we can obtain the existence of both a solution and an attractor for truncated system \eqref{jieduangl2}..
\begin{theorem}\label{dingli4.1}
	For any $\varepsilon \in (0,\varepsilon^{*}]$ and $u_{0}^{\varepsilon,m} \in B_{r^{*}}^{m}$ with $m\in\NN$, the truncated system \eqref{jieduangl2} possesse a unique solution $u_{n}^{\varepsilon,m} \in 
	B_{r^{*}}^{m}$ for any $n\in\NN$. Furthermore,  the solution semigroup has a unique finite dimensional numerical attractor ${\mathcal{A}}_{m}^{\varepsilon}$(also referred to as the truncated numerical attractor), as described below:
	\begin{equation}\label{jieduanshuzhixiyinzi}
		{\mathcal{A}}_{m}^{\varepsilon} = \bigcap_{N=1}^{\infty}\overline{\bigcup_{n=N}^{\infty}u_{n}^{\varepsilon,m}(B_{r^{*}}^{m})}.
	\end{equation}
\end{theorem}
\begin{proof}[Proof]
	Using the same method as in Theorem \ref{dingli2.1}, we have
	\begin{equation*}
		\|u_{n}^{\varepsilon,m}\|^2 \leq\frac{1}{1+(2\gamma-8\lambda)\varepsilon}(\|u_{n-1}^{\varepsilon,m}\|^2+2\varepsilon c_1).
	\end{equation*}
	According to Theorem \ref{dingli2.1}, the truncated system \eqref{jieduangl2} has a unique solution $u_{n}^{\varepsilon,m}(u_{0}^{\varepsilon,m})\in B_{r^{*}}^{m}$ for any $\varepsilon\in(0,\varepsilon^{*}], n\in\NN$ and $u_{0}^{\varepsilon,m}\in B_{r^{*}}^{\varepsilon,m}$. Furthermore, the solution semigroup on $B_{r^{*}}^{\varepsilon,m}$ is compact and absorbing in the finite dimensional 
	space. Thus,the solution semigroup possesses a unique attractor ${\mathcal{A}}_{m}^{\varepsilon}$ as defined in \eqref{jieduanshuzhixiyinzi}.
\end{proof}
\subsection{Convergence from the truncated numerical attractor to the numerical attractor}
In this subsection, we will prove the convergence of the truncated numerical attractor ${\mathcal{A}}_{m}^{\varepsilon}$ as $m\to\infty$. At first, we give the Lemma to show that the tail of any element in ${\mathcal{A}}_{m}^{\varepsilon}$ becomes uniformly small as $m\to\infty$.
\begin{lemma}\label{yinli4.1}
	Suppose $\varepsilon\in(0,\varepsilon^{*}]$. For any $\delta\textgreater0$, there exists $I_{\delta}\in\NN$ such that
	\begin{equation}
		\sum_{I_{\delta}\leq|j|\leq m}|z_{j}|^{2} \leq \delta, \ \forall z = (z_{j})_{|j|\leq m} \in {\mathcal{A}}_{m}^{\varepsilon},\ \forall m\geq I_{\delta}, m\in \NN.
	\end{equation}
\end{lemma}
\begin{proof}[Proof]
	We difine $\xi_{l}^{m} = (\xi_{l,j})_{|j|\leq m}$ as the finite dimensional cut-off function. Taking the inner product of \eqref{jieduangl2} with ${\xi_{l}^{m}}{u_{n}^{\varepsilon,m}}$, then we take the real part:
		\begin{equation}\label{jieduanlisanweibuguji1}
			\begin{split}
			\sum_{|j|\leq m}\xi_{l,j}|u_{n,j}^{\varepsilon,m}|^2 &= \operatorname{Re}(u_{n-1}^{\varepsilon,m},\xi_{l}^{m}u_{n}^{\varepsilon,m}) + \varepsilon\operatorname{Re}\left(\left(\lambda + i\mu\right)\Lambda u_{n}^{\varepsilon,m}, \xi_{l}^{m}u_{n}^{\varepsilon,m}\right)\\
			&\quad- \varepsilon\gamma\sum_{|j|\leq m}\xi_{l,j}\left|u_{n,j}^{\varepsilon,m}\right|^{2} - \varepsilon k \sum_{|j|\leq m}\xi_{l,j}|u_{n,j}^{\varepsilon,m}|^{p+2} + \varepsilon\operatorname{Re}\left(g^{m}, \xi_{l}^{m}u_{n}^{\varepsilon,m}\right).
			\end{split}
		\end{equation}
	According to Theorem \ref{dingli4.1}, we know that $|u_{n,j}^{\varepsilon,m}|\leq\| u_{n,j}^{\varepsilon,m} \| \leq r^{*}$, and by Young's inequality, we have
	\begin{equation}\label{jieduandiyixiang}
	\begin{split}
		\operatorname{Re}\left[\lambda(\Lambda_{m}u_{n}^{\varepsilon,m}, \xi_{l}^{m}u_{n}^{\varepsilon,m})\right] &\leq \lambda \xi_{l,-m}\left(3|u_{n,-m}^{\varepsilon,m}|^{2} + \frac{1}{2}|u_{n,m}^{\varepsilon,m}|^{2} + \frac{1}{2}|u_{n,-m+1}^{\varepsilon,m}|^{2}\right) + \lambda \sum_{|j|\leq m-1} \xi_{l,j}\left(3|u_{n,j}^{\varepsilon,m}|^{2} \right.\\
		&\quad\left. + \frac{1}{2}|u_{n,j-1}^{\varepsilon,m}|^{2} + \frac{1}{2}|u_{n,j+1}^{\varepsilon,m}|^{2}\right) + \lambda\xi_{l,m}\left(3|u_{n,m}^{\varepsilon,m}|^{2} + \frac{1}{2}|u_{n,m-1}^{\varepsilon,m}|^{2} + \frac{1}{2}|u_{n,-m}^{\varepsilon,m}|^{2}\right).
	\end{split}
	\end{equation}
	When $j = -m$, the right-hand side of \eqref{jieduandiyixiang} is expressed as:
	\begin{equation*}
		\begin{split}
			&\quad3\lambda\xi_{l,-m}|u_{n,-m}^{\varepsilon,m}|^{2} + \frac{1}{2}\lambda\xi_{l,-m+1}|u_{n,-m}^{\varepsilon,m}|^{2} + \frac{1}{2}\lambda\xi_{l,m}|u_{n,-m}^{\varepsilon,m}|^{2} \\
			&= \frac{7}{2}\lambda\xi_{l,-m}|u_{n,-m}^{\varepsilon,m}|^{2} + \frac{1}{2}\lambda\xi_{l,-m}|u_{n,-m}^{\varepsilon,m}|^{2} + \frac{1}{2}\lambda\left(\xi_{l,-m+1}|u_{n,-m}^{\varepsilon,m}|^{2} - \xi_{l,-m}|u_{n,-m}^{\varepsilon,m}|^{2}\right)\\
			&\leq 4\lambda\xi_{l,-m}|u_{n,-m}^{\varepsilon,m}|^{2} + \lambda \frac{c_{2}}{2l}|u_{n,-m}^{\varepsilon,m}|^{2}. 
		\end{split}
	\end{equation*}
	When $j = -m+1$, the right-hand side of \eqref{jieduandiyixiang} is expressed as:
	\begin{equation*}
	\begin{split}
		&\quad3\lambda\xi_{l,-m+1}|u_{n,-m+1}^{\varepsilon,m}|^{2} + \frac{1}{2}\lambda\xi_{l,-m}|u_{n,-m-1}^{\varepsilon,m}|^{2} + \frac{1}{2}\lambda\xi_{l,-m+2}|u_{n,-m-1}^{\varepsilon,m}|^{2} \\
		& = (3+\frac{1}{2}+\frac{1}{2})\lambda\xi_{l,-m+1}|u_{n,-m+1}^{\varepsilon,m}|^{2} + \frac{1}{2}\lambda\left( \xi_{l,-m} - \xi_{l,-m+1}  \right)|u_{n,-m+1}^{\varepsilon,m}|^{2} + \frac{1}{2}\lambda\left( \xi_{l,-m+2} - \xi_{l,-m+1}  \right)|u_{n,-m+1}^{\varepsilon,m}|^{2}\\
		&\leq 4\lambda\xi_{l,-m+1}|u_{n,-m+1}^{\varepsilon,m}|^{2} + \lambda \frac{c_{2}}{l}|u_{n,-m+1}^{\varepsilon,m}|^{2}.
	\end{split}
	\end{equation*}
	When $j = -m+2,\dots,m+2$, the right-hand side of \eqref{jieduandiyixiang} is expressed as:
	\begin{equation*}
	\begin{split}
		&\quad3\lambda\xi_{l,j}|u_{n,j}^{\varepsilon,m}|^{2} + \frac{1}{2}\lambda\xi_{l,j-1}|u_{n,j}^{\varepsilon,m}|^{2} + \frac{1}{2}\lambda\xi_{l,j+1}|u_{n,j}^{\varepsilon,m}|^{2} \\
		& = (3+\frac{1}{2}+\frac{1}{2})\lambda\xi_{l,j}|u_{n,j}^{\varepsilon,m}|^{2} + \frac{1}{2}\lambda\left( \xi_{l,j-1} - \xi_{l,j}  \right)|u_{n,j}^{\varepsilon,m}|^{2} + \frac{1}{2}\lambda\left( \xi_{l,j+1} - \xi_{l,j}  \right)|u_{n,j}^{\varepsilon,m}|^{2}\\
		&\leq 4\lambda\xi_{l,j}|u_{n,j}^{\varepsilon,m}|^{2} +  \lambda \frac{c_{2}}{l}|u_{n,j}^{\varepsilon,m}|^{2}.
	\end{split}
	\end{equation*}
	When $j = m-1$, the right-hand side of \eqref{jieduandiyixiang} is expressed as:
	\begin{equation*}
	\begin{split}
		&\quad3\lambda\xi_{l,m-1}|u_{n,m-1}^{\varepsilon,m}|^{2} + \frac{1}{2}\lambda\xi_{l,m-2}|u_{n,m-1}^{\varepsilon,m}|^{2} + \frac{1}{2}\lambda\xi_{l,m}|u_{n,m-1}^{\varepsilon,m}|^{2} \\
		& = (3+\frac{1}{2}+\frac{1}{2})\lambda\xi_{l,m-1}|u_{n,m-1}^{\varepsilon,m}|^{2} + \frac{1}{2}\lambda\left( \xi_{l,m-2} - \xi_{l,m-1}  \right)|u_{n,m-1}^{\varepsilon,m}|^{2} + \frac{1}{2}\lambda\left( \xi_{l,m} - \xi_{l,m-1}  \right)|u_{n,m-1}^{\varepsilon,m}|^{2}\\
		&\leq 4\lambda\xi_{l,m-1}|u_{n,m-1}^{\varepsilon,m}|^{2} + \lambda\frac{c_{2}}{l}|u_{n,m-1}^{\varepsilon,m}|^{2}.
	\end{split}
	\end{equation*}	
	When $j = m$, the right-hand side of \eqref{jieduandiyixiang} is expressed as:
	\begin{equation*}
		\begin{split}
			&\quad3\lambda\xi_{l,m}|u_{n,m}^{\varepsilon,m}|^{2} + \frac{1}{2}\lambda\xi_{l,-m}|u_{n,m}^{\varepsilon,m}|^{2} + \frac{1}{2}\lambda\xi_{l,m-1}|u_{n,m}^{\varepsilon,m}|^{2} \\
			&= \frac{7}{2}\lambda\xi_{l,m}|u_{n,m}^{\varepsilon,m}|^{2} + \frac{1}{2}\lambda\xi_{l,m}|u_{n,m}^{\varepsilon,m}|^{2} + \frac{1}{2}\lambda\left(\xi_{l,m-1}|u_{n,m}^{\varepsilon,m}|^{2} - \xi_{l,m}|u_{n,m}^{\varepsilon,m}|^{2}\right)\\
			&\leq 4\lambda\xi_{l,m}|u_{n,m}^{\varepsilon,m}|^{2} + \lambda \frac{c_{2}}{2l}|u_{n,m}^{\varepsilon,m}|^{2}. 
		\end{split}
	\end{equation*}
	Therefore, we rewrtie the \eqref{jieduandiyixiang} as follows:
	\begin{equation}\label{jieduandiyixiangxin}
		\operatorname{Re}\left[\lambda(\Lambda_{m}u_{n}^{\varepsilon,m}, \xi_{l}^{m}u_{n}^{\varepsilon,m})\right] \leq 4\lambda\sum_{|j|\leq m}\xi_{l,j}|u_{n,j}^{\varepsilon,m}|^{2} + \lambda \frac{c_{2}}{l}(r^{*})^{2}.
	\end{equation}
	Using a same method from \eqref{diyixiang2}, we have
	\begin{equation}\label{jieduandierxiang}
	\begin{split}
			&\quad\operatorname{Re}\left[  i\mu(\Lambda_{m}u_{n}^{\varepsilon,m}, \xi_{l}^{m}u_{n}^{\varepsilon,m})  \right] \\
			&= 	\operatorname{Re}\left[ i\mu \sum_{|j|\leq m}\xi_{l,j}|D^{+}_{m}u_{j}^{\varepsilon,m}|^{2} + i\mu(D^{+}_{m}u_{n}^{\varepsilon,m},u_{n}^{\varepsilon,m}D^{+}_{m}\xi_{l}^{m})    \right]\\
			& = \operatorname{Re}\left[ i\mu(D^{+}u,uD^{+}\xi_{l})  \right]
			 \leq \mu \| D^{+}_{m}u^{\varepsilon,m}_{n} \| \| u_{n}^{\varepsilon,m} \| \| D^{+}_{m}\xi_{l}^{m} \|_{\infty}
			 \leq 2|\mu| \frac{c_{2}}{l}\| u_{n}^{\varepsilon,m} \|^{2} 
			\leq 2|\mu|\frac{c_{2}}{l}(r^{*})^{2} .
	\end{split}
	\end{equation}
	Besides, we have
	\begin{equation}\label{jieduandisanxiang}
		\operatorname{Re}\left[ (u_{n-1}^{\varepsilon,m},\xi_{l}^{m}u_{n}^{\varepsilon,m}) \right] \leq \frac{1}{2} \sum_{|j|\leq m}\xi_{l,j}|u_{n,j}^{\varepsilon,m}|^{2} + \frac{1}{2} \sum_{|j|\leq m}\xi_{l,j}|u_{n-1,j}^{\varepsilon,m}|^{2}
	\end{equation}
	Submiting \eqref{jieduandiyixiangxin}-\eqref{jieduandisanxiang} to the \eqref{jieduanlisanweibuguji1}, similarly with Theorem \ref{dingli3.2} we obtain
	\begin{equation}\label{jieduanlisanweibuguji2}
	\begin{split}
		\sum_{|j|\leq m}\xi_{l,j}|u_{n,j}^{\varepsilon,m}|^2 &\leq \left[ \frac{1}{2}+\varepsilon(4\lambda-\gamma )  \right] \sum_{|j|\leq m}\xi_{l,j}|u_{n,j}^{\varepsilon,m}|^{2} + \frac{1}{2} \sum_{|j|\leq m}\xi_{l,j}|u_{n-1,j}^{\varepsilon,m}|^{2} + \varepsilon\left[ (\lambda+2|\mu|)\frac{c_{2}}{l}(r^{*})^{2} \right.\\
		&\quad \left. + \frac{p+1}{p+2}\left[ \frac{1}{k(p+2)} \right]^{\frac{1}{p+1}} \sum_{|j|\leq m} \xi_{l,j}|g_{j}|^{\frac{p+2}{p+1}}  \right].
	\end{split}
	\end{equation}
	The last term of the \eqref{jieduanlisanweibuguji2} tends to zero as $l \to \infty$, thus we choose $l_{\delta}$, $\forall l\geq l_{\delta}$, we have
	\begin{equation*}
		 (\lambda+2|\mu|)\frac{c_{2}}{l}(r^{*})^{2}  + \frac{p+1}{p+2}\left[ \frac{1}{k(p+2)} \right]^{\frac{1}{p+1}} \sum_{|j|\leq m} \xi_{l,j}|g_{j}|^{\frac{p+2}{p+1}}   \leq \frac{\delta}{2} (\gamma-4\lambda).
	\end{equation*}
	We can rewrite the \eqref{jieduanlisanweibuguji2}
	\begin{equation}\label{jieduanlisanweibuguji3}
		\sum_{|j| \leq m}\xi_{l,j}|u_{n,j}^{\varepsilon,m}|^2 \leq \frac{2\varepsilon(\gamma-4\lambda)}{1+2\varepsilon(\gamma-4\lambda)}\frac{\delta}{2} +\frac{1}{1+2\varepsilon(\gamma-4\lambda)}\sum_{|j|\leq m}\xi_{l,j}|u_{n-1,j}^{\varepsilon,m}|^{2}.
	\end{equation}
	Hence, there exists $N_{\delta}$, and we obtain the following:
	\begin{equation*}
		\sum_{|j|\leq m}\xi_{l,j}|u_{n,j}^{\varepsilon,m}|^{2} \leq \delta, \ \forall l \geq l_{\delta} \ and \ n \geq N_{\delta}.
	\end{equation*}
	Specially, for $I_{\delta} = 2l_{\delta}$, we get
	\begin{equation*}
		\sum_{I_{\delta}\leq |j| \leq m} | u_{N_{\delta},j}^{\varepsilon,m}(y) |^{2} \textless \delta, \ \forall y\in B_{r^{*}}^{m}.
	\end{equation*}
	For any $\varepsilon\in(0,\varepsilon^{*}]$ and $m\in\NN$, suppose $z\in{\mathcal{A}}^{\varepsilon}_{m} \subset B_{r^{*}}^{m} $, we obtain $z = u_{N_{\delta}}^{\varepsilon,m}(y)$. Therefore,
	\begin{equation*}
		\sum_{I_{\delta}\leq |j| \leq m}|z_{j}|^{2} = \sum_{I_{\delta}\leq |j| \leq m} |u_{N_{\delta}}^{\varepsilon,m}(y)|^{2} \textless \delta.
	\end{equation*}
	The lemma is proven.
\end{proof}
The next lemma follows from \cite{Temam-1997}.
\begin{lemma}\label{yinli4.2}
	$z\in {\mathcal{A}}^{\varepsilon}_{m}$ if and only if there exists a bounded solution $u_{n}^{m}$ of the truncated system \eqref{jieduangl2} such that $u_{0}^{m}=z$, while $y\in{\mathcal{A}}^{\varepsilon}$  if and only if there exists a bounded solution $u_{n}$ of the IES \eqref{glt} in $\ell^{2}$ such that $u_{0} = y$.
\end{lemma}

To prove the convergence result of ${\mathcal{A}}_{m}^{\varepsilon}$ under the Hausdorff semi-distance, observe that any element in $\CC^{2m+1}$ can be naturally expanded to an element in $\ell^{2}$. The null-expansion of a point $z\in\CC^{2m+1}$ can be defined as follows:
\begin{equation*}
	\widetilde{z}_{j} = 0,\ \forall |j| \textgreater m;\, \  \widetilde{z}_{j} = z_{j},\ \forall|j|\leq m.
\end{equation*}
From this viewpoint, the truncated attractor ${\mathcal{A}}^{\varepsilon}_{m}$ naturally extends to a null-expansion $\widetilde{{ \mathcal{A}}^{\varepsilon}_{m}} \subset \ell^{2}$.
\begin{theorem}\label{dingli4.2}
	Suppose $\varepsilon\in(0,\varepsilon^{*}]$. As $m\to\infty$, the truncated attractor ${\mathcal{A}}^{\varepsilon}_{m}$ in \eqref{jieduanshuzhixiyinzi} upper semi-converges to the numerical attractor ${\mathcal{A}}^{\varepsilon}$ of the IES \eqref{glt} under the Hausdorff semi-distance, i.e., 
	\begin{equation}\label{xiyinzibijin1}
		\lim_{m\to\infty}dist_{\ell^{2}}({\mathcal{A}}^{\varepsilon}_{m},{\mathcal{A}}^{\varepsilon}) = \lim_{m\to\infty}dist_{\ell^{2}}(\widetilde{{ \mathcal{A}}^{\varepsilon}_{m}},{\mathcal{A}}^{\varepsilon}) = 0.
	\end{equation}
\end{theorem}
\begin{proof}[Proof]
	We proof this theorem by contradiction. Assuming that \eqref{xiyinzibijin1} is false, there exists a constant $M_{0}\textgreater0$, a subsequence $m_{l}$ and $z_{m_{l}}\in{\mathcal{A}}_{m_{l}}^{\varepsilon}$ such that
	\begin{equation}\label{xiyinzibijin2}
		{\rm d}(z^{m_{l}},{\mathcal{A}}^{\varepsilon}) = {\rm d}(\widetilde{z^{m_{l}}},{\mathcal{A}}^{\varepsilon}) \geq M_{0}, \ \forall l\in\NN.
	\end{equation}
	Since $z^{m_l} \in \mathscr{A}_{m_l}^\varepsilon$, the unique solution $u_n^{\varepsilon,m_l} = u_n^{\varepsilon,m_l}(z^{m_l})$ of the truncated system \eqref{jieduangl2} exists by Lemma \ref{dingli4.2}, and it belongs to $\mathscr{A}_{m_l}^\varepsilon$. The null-expansion of $u_n^{\varepsilon,m_l}$ is denoted by $\widetilde{u_n^{\varepsilon,m_l}}$. By Lemma \ref{dingli4.1}, for any $\delta \textgreater 0$, there exists $I_\delta \in \mathbb{N}$ such that
	\[
	\sum_{|j| \geq I(\delta)} \left| \widetilde{u_{n,j}^{\varepsilon,m_l}} \right|^2 = \sum_{I(\delta) \leq |j| \leq m_l} \left| u_{n,j}^{\varepsilon,m_l} \right|^2 < \delta, \quad \forall n \in \mathbb{Z}, \, l \in \mathbb{N}.
	\]
	
	According to Theorem \ref{dingli4.1}, the attractor is contained within a ball of radius $r^*$. Therefore, the sequence $\left( u_{n,j}^{\varepsilon,m_l} \right)_{|j| < I_\delta}$ is bounded in $\mathbb{R}^{2I_\delta - 1}$, and the null-expansion $\widetilde{u_n^{\varepsilon,m_l}}$ is relatively compact in $\ell^2$. There exists a subsequence of $\{l\}$ (which we still denote by $l$) and $u_n^* \in \ell^2$ such that
	\[
	\lim_{l \to \infty} \left\| \widetilde{u_n^{\varepsilon,m_l}} - u_n^* \right\| = 0, \quad \forall n \in \mathbb{Z}. \tag{4.14}
	\]
	
	Next, we will prove that $u_n^*$ is the solution of the IES \eqref{glt}. Since $u_n^{\varepsilon,m_l}$ is the solution of the truncated system \eqref{jieduangl2}, for any $n \in \mathbb{Z}$ and $l \in \mathbb{N}$, we have
	\begin{equation*}
		u_{n}^{\varepsilon,m_{l}} = u_{n-1}^{\varepsilon,m_{l}} + \varepsilon\left( (\lambda+i\mu)\Lambda_{m_{l}}u_{n}^{\varepsilon,m_{l}} - (\gamma+i\beta)u_{n}^{\varepsilon,m_{l}} - (k+i\nu)|u_{n}^{\varepsilon,m_{l}}|^{p}u_{n}^{\varepsilon,m_{l}} + g^{m_{l}} \right)
	\end{equation*}
	According to the definition of the $m_{l}$-truncation operator, for each component $j\in\ZZ$ as $m_{l}\to\infty$, we have
	\begin{equation*}
		(\Lambda_{m_{l}}u_{n}^{\varepsilon,m_{l}})_{j} = (\Lambda\widetilde{u_n^{\varepsilon,m_l}})_{j},\ u_{n,j}^{\varepsilon,m_{l}} = \widetilde{u}_{n,j}^{\varepsilon,m_{l}},\ (g^{m_{l}})_{j} = g_{j},\ \forall n\in\ZZ,
	\end{equation*}
	here $\widetilde{u}_{n,j}^{\varepsilon,m_{l}}$ represents the $j$-component of the null-expansion $\widetilde{u_{n}^{\varepsilon,m_{l}}}$. As a result, it satisfies
	\begin{equation}\label{jieduankuochong}
		\widetilde{u}_{n,j}^{\varepsilon,m_{l}} = \widetilde{u}_{n-1,j}^{\varepsilon,m_{l}} + \varepsilon\left( (\lambda+i\mu)(\Lambda\widetilde{u_{n}^{\varepsilon,m_{l}}})_{j} - (\gamma+i\beta)\widetilde{u}_{n,j}^{\varepsilon,m_{l}} - (k+i\nu)|\widetilde{u}_{n,j}^{\varepsilon,m_{l}}|^{p}\widetilde{u}_{n,j}^{\varepsilon,m_{l}} + g_{j} \right)
	\end{equation}
	As $l\to\infty$(with $m_{l}\to\infty$ as well), from \eqref{jieduankuochong} we have
	\begin{equation*}
		u_{n,j}^{*} = u_{n-1,j}^{*} + \varepsilon\left( (\lambda+i\mu)\Lambda u_{n,j}^{*} - (\gamma+i\beta)u_{n,j}^{*} - (k+i\nu)|u_{n,j}^{*}|^{p}u_{n,j}^{*} + g_{j} \right)\  \forall j,z\in\ZZ
	\end{equation*}
	This indicates that $u_n^*$ is the solution of the IES \eqref{glt}. According to Theorem \ref{dingli4.1}, it follows that $\widetilde{u_n^{\varepsilon,m_l}} \in B_{r^*}$. From \eqref{jieduankuochong}, we get
	\[
	\|u_n^*\| = \lim_{l \to \infty} \|\widetilde{u_n^{\varepsilon,m_l}}\| \leq r^*, \quad \forall n \in \mathbb{Z},
	\]
	which implies that the solution $u_n^*$ of the IES \eqref{jieduankuochong} is bounded. According to Lemma \ref{yinli4.2}, It follows that $u_0^* \in \mathscr{A}^\varepsilon$ and
	\[
	\lim_{l \to \infty} \widetilde{z^{m_{l}}} = \lim_{l \to \infty} \widetilde{u_0^{\varepsilon,m_l}(z^{m_{l}})} = u_0^*.
	\]
	This contradicts the assumption in \eqref{xiyinzibijin2}, thus proving the theorem.
\end{proof}
\subsection{Bounds and continuity properties of numerical attractors}
In this part, We will derive the bounds and continuity properties of the numerical attractor $\mathscr{A}^\varepsilon$ and the truncated numerical attractor $\mathscr{A}_m^\varepsilon$ with under the Hausdorff semi-distance. Let
\[
\rho(A,B) = \max\{d(A,B), d(B,A)\}, \quad \|A\| = \rho(A,\{0\}), \quad \forall A, \, B \subset \ell^2 \text{ or } \mathbb{R}^{2m+1}.
\]

The two attractors will be denoted as $\mathscr{A}^\varepsilon(g, \gamma, \lambda)$ and $\mathscr{A}_m^\varepsilon(g, \gamma, \lambda)$, which depend on the external force $g \in \ell^2$ and the damping constant $\gamma-4\lambda \textgreater 0$.
\begin{theorem}\label{dingli4.3}
(i) For any $\varepsilon \in (0, \varepsilon^*]$ and $m \in \mathbb{N}$, the following upper bounds hold for the two attractors:
\begin{equation}\label{yitiao}
\begin{split}
	&\|\mathscr{A}^\varepsilon(g, \gamma, \lambda)\| \leq \sqrt{\frac{1}{\gamma-4\lambda}\frac{p+1}{p+2}\frac{1}{\left[k(p+2)\right]^{\frac{1}{p+1}}}\|g\|_{\frac{p+2}{p+1}}^{\frac{p+2}{p+1}}}, \\
	&\|\mathscr{A}_m^\varepsilon(g, \gamma, \lambda)\| \leq \sqrt{\frac{1}{\gamma-4\lambda}\frac{p+1}{p+2}\frac{1}{\left[k(p+2)\right]^{\frac{1}{p+1}}}\|g^{m}\|_{\frac{p+2}{p+1}}^{\frac{p+2}{p+1}}}. 
\end{split}
\end{equation}

In particular, if $g = 0$, both attractors reduce to a point, i.e.,
\[
\mathscr{A}^\varepsilon(0, \gamma, \lambda) = \mathscr{A}_m^\varepsilon(0, \gamma, \lambda) = \{0\}. 
\]

(ii) The following continuity properties hold for the two attractors:
\[
\lim_{g \to 0} \rho\left( \mathscr{A}^\varepsilon(g, \gamma,\lambda), \{0\} \right) = 0, \quad \lim_{g \to 0} \rho\left( \mathscr{A}_m^\varepsilon(g, \gamma, \lambda,|\mu|), \{0\} \right) = 0,
\]
\[
\lim_{\gamma-4\lambda \to +\infty} \rho\left( \mathscr{A}^\varepsilon(g, \gamma, \lambda), \{0\} \right) = 0, \quad \lim_{\gamma-4\lambda \to +\infty} \rho\left( \mathscr{A}_m^\varepsilon(g, \gamma, \lambda,|\mu|), \{0\} \right) = 0. 
\]
\end{theorem}
\begin{proof}[Proof]
	At first, we define $c_{1}^{m}$:
	\begin{equation*}
		c_{1}^{m} = \frac{p+1}{p+2}\frac{1}{\left[k(p+2)\right]^{\frac{1}{p+1}}}\|g^{m}\|_{\frac{p+2}{p+1}}^{\frac{p+2}{p+1}}
	\end{equation*}
	(i) We prove the second inequality in \eqref{yitiao}, and the first inequality follows using a similar method. For any $\eta\textgreater0$, let $r^{0}$ be a radius that satisfies
	\begin{equation*}
		\sqrt{\frac{c_{1}^{m}}{\gamma-4\lambda}} \textless r_{0} \leq \sqrt{\eta + \frac{c_{1}^{m}}{\gamma-4\lambda}} := r^{*,m} \leq \sqrt{\eta+\frac{c_{1}}{\gamma-4\lambda}} = r^{*}.
	\end{equation*}
	Next, we prove that the ball $B_{r_{0}}^{m}$ in $\CC^{2m+1}$ is absorbing for the truncated system \eqref{jieduangl2}. Taking the inner product of \eqref{jieduangl2} with $u_{n}^{\varepsilon,m}$, then we take the real part 
	\begin{equation*}
		\| u_{n}^{\varepsilon,m} \|^{2} = \operatorname{Re}(u_{n-1}^{\varepsilon,m},u_{n}^{\varepsilon,m}) + \varepsilon\operatorname{Re} \left[    (\lambda + i\mu ) (\Lambda_{m} u_n^{\varepsilon,m}, u_n^{\varepsilon,m})  \right]  - \varepsilon \gamma \|u_n^{\varepsilon,m}\|^2
		-\varepsilon k \|u_n^{\varepsilon,m}\|_{p+2}^{p+2} + \varepsilon \operatorname{Re}(g^{m}, u_n^{\varepsilon,m}).
	\end{equation*}
	Based on the forms of the matrices $\Lambda_{m}$ in \eqref{lambdam}, similarly with \eqref{jieduandiyixiangxin} and \eqref{jieduandierxiang}, we get
	\begin{equation*}
		\operatorname{Re} \left[    (\lambda + i\mu ) (\Lambda_{m} u_n^{\varepsilon,m}, u_n^{\varepsilon,m})  \right] \leq \operatorname{Re}\left[(\lambda+i\mu)\|D_{m}^{+}u_{n}^{\varepsilon,m}\|^{2}\right]              \leq 4\lambda\| u_{n}^{\varepsilon,m} \|^{2},
	\end{equation*}
	then we use a similar method with \eqref{ss1}, we get
	\begin{equation}\label{jieduanguji}
		\| u_{n}^{\varepsilon,m} \|^{2} \leq \frac{1}{2}\| u^{\varepsilon,m}_{n} \|^{2} + \frac{1}{2}\| u_{n-1}^{\varepsilon,m} \|^{2} +  \varepsilon(4\lambda-\gamma)\|u_{n}^{\varepsilon,m}\|^{2} +\varepsilon c_{1}^{m}. 
	\end{equation}
	Therefore, for any initial $u_{0}\in B_{r}^{m}$ with $0\textless r \textless r^{*,m}$, we can rewrite \eqref{jieduanguji} as follows:
	\begin{equation*}
	\begin{split}
		\| u_{n}^{\varepsilon,m}(u_{0}) \|^{2} &\leq \frac{1}{1+2\varepsilon(\gamma-4\lambda)}\|u_{n-1}^{\varepsilon,m}\|^{2} + \frac{2+\varepsilon c_{1}^{m}}{1+2\varepsilon(\gamma-4\lambda)}\\
		&\leq  \frac{1}{\left[ 1+2\varepsilon(\gamma-4\lambda) \right]^{n}} \| u_{0}^{\varepsilon,m}\|^{2} + c_{1}^{m}\sum_{j=1}^{n} \frac{2\varepsilon}{\left[ 1+2\varepsilon(\gamma-4\lambda) \right]^{j}}\\
		&\leq \frac{r^{2}}{\left[ 1+2\varepsilon(\gamma-4\lambda) \right]^{n}} + c_{1}^{m}. 
	\end{split}
	\end{equation*}
	As mentioned before, $\sqrt{\frac{c_{1}^{m}}{\gamma-4\lambda}} \textless r_{0}$. For any $r\in(0,r^{*,m}]$, there exists $N=N(r)$ such that for any $n\geq N$ and $u_{0}\in B_{r}^{*}$, we have
	\begin{equation*}
		\|u_{n}^{\varepsilon,m}(u_{0})\|^{2} \leq r_{0}^{2}.
	\end{equation*}
	Therefore, $B_{r_{0}}^{*}$ is bounded and absorbing for the truncated system \eqref{jieduangl2}.
	
	Since an attractor is contained with any absorbing set, we have
	\begin{equation*}
		{\mathcal{A}}_{m}^{\varepsilon}(g,\gamma,\lambda,|\mu|) \subset B_{r_{0}}^{m} \Longrightarrow \| {\mathcal{A}}^{\varepsilon}_{m}(g,\gamma,\lambda,|\mu|) \|^{2}, \ \forall r_{0} \in \left( c_{1}^{m}, r^{*,m}  \right], \ \forall\varepsilon \in \left(  0,\varepsilon^{*,m} \right].
	\end{equation*}
	As $r_{0} \to c_{1}^{m}$, we get
	\begin{equation*}
		\| {\mathcal{A}}_{m}^{\varepsilon}(g,\gamma,\lambda,|\mu|) \| \leq c_{1}^{m},\ \forall \varepsilon\in \left( 0,\varepsilon^{*} \right],\ m\in\NN.
	\end{equation*}
\end{proof}
The case when $g=0$ and Conclusion (ii) are directly derived from Conclusion(i).
\section{Finite dimensional approximation of random attrator}
The stochastic of Ginzburg-Landau lattice system \eqref{gl} can be written as follows:
\begin{equation}\label{suijigl}
\begin{cases}
	&\frac{d u(t)}{d t}=(\lambda+i\mu)\Lambda u-(\gamma+i\beta)u-(k+i\nu)|u|^{p}u+g + a u \circ \frac{dW}{dt},\\
	&u(0)=u_0.
\end{cases}
\end{equation}
where $a \textgreater 0$ denotes the noise intensity, $\circ$ represents the Stratonovich stochastic differential, and $W(t)$ is a Wiener process defined on a complete filtered probability space $(\Omega,\mathcal{F},\left\{\mathcal{F}_{t}\right\}_{t\in\mathbb{R}},\mathbb{P})$.
\subsection{Existence, upper semi-convergence and upper semi-continuity of random attractor}
Let $W(t,\omega)(t\in\RR)$ denote the standard one-dimensional, two-sided Wiener process with sample paths $\omega(t)$ in the classical Wiener space $(\Omega,\mathcal{F},\mathbb{P},\left\{\theta_{t}\right\}_{t\in\RR})$, where
\begin{equation*}
	\Omega = \left\{ \omega\in C (\CC,\CC):\ \omega(0)=0 \right\},
\end{equation*}
and the Borel $\sigma$-algebra $\mathcal{F}$ is generated by the compact-open topology. The shift operator is given by $\theta_{t}\omega(\cdot) = \omega(t+\cdot) - \omega(t)$, as detailed in \cite{Bates-2009,Li-2015}. We consider the It\^o equation $dz + zdt = dW(t) $, thus we get the stochastic Ornstein-Uhlenbeck process
\begin{equation*}
	z(\theta_{t}\omega) := -\int_{-\infty}^{0} {\rm{e}}^{s} (\theta_{t}\omega)(s)ds = - \int_{-\infty}^{0} {\rm{e}}^{s} \omega(t+s)ds+\omega(t),\ t\in\RR,\ \omega\in\Omega.
\end{equation*}
More properties of the Ornstein-Uhlenbeck process can be found in \cite{Aronld-1998}. $z(\theta_{t}\omega)$ is the real-valued path-continuous Ornstein Uhlenbeck process with the following limits:
\begin{equation}\label{OUxingzhi}
	\lim_{t\to\pm\infty} \frac{|z(\theta_{t}\omega)|}{|t|} = 0 \ \text{and}\ \lim_{t\to\pm\infty} \frac{1}{t} \int_{0}^{t} z(\theta_{t}\omega)dt = 0.
\end{equation}

To analyze the pathwise dynamics of the stochastic Ginzburg-Landau system, we employ the Ornstein-Uhlenbeck transform to convert the original stochastic equation into a random PDE. Let $u$ be the solution of \eqref{suijigl} and define
\begin{equation*}
	U^{a}(t,\tau,\omega,U_{\tau}) = {\rm{e}}^{-az(\theta_{t}\omega)} u(t,\tau,\omega,u_{\tau}),\ \text{wiht}\ U_{\tau} = {\rm{e}}^{-az(\theta_{t}\omega)}u_{\tau}, 
\end{equation*}
then, we have
\begin{equation}\label{daihuan}
	dU^{a} = {\rm{e}}^{-az(\theta_{t}\omega)}du - a{\rm{e}}^{-az(\theta_{t}\omega)}u\circ dz(\theta_{t}\omega).
\end{equation}
Based on \eqref{suijigl} and \eqref{daihuan}, we consider the equation with random coefficients as follows:
\begin{equation}\label{suijigl2}
	\frac{dU^{a}}{dt} = (\lambda+i\mu)\Lambda U^{a} - (\gamma+i\beta)U^{a} - (k+i\nu)(|U^{a}|^{p})U^{a} {\rm{e}}^{paz(\theta_{t}\omega)} + {\rm{e}}^{-az(\theta_{t}\omega)}g + aU^{a}z(\theta_{t}\omega),
\end{equation}
the initial condition is 
\begin{equation*}
	U_{0} = {\rm{e}}^{-az(\theta_{t}\omega)}u_{0}.
\end{equation*}
Let $\Phi^{a}$ denote the continuous random dynamical system associated with the problem in \eqref{suijigl2}, which can be expressed as:
\begin{equation*}
	\Phi^{a}(\cdot,\omega,\Phi_{0}) = U^{a}(\cdot,\omega,U_{0})\in C([0,+\infty),\ell^{2});
\end{equation*}
Thus  generates a random dynamical system $\Gamma^{a}:\RR^{+}\times\Omega\times\ell^{2}\to\ell^{2}$ given by
\begin{equation}\label{suijixitong}
	\Gamma^{a}(t,\omega) \Phi_{0}=\Phi^{a}(t,\omega,\Phi_{0}),\ \forall(t,\omega,\Phi_{0})\in \RR^{+}\times\Omega\times\ell^{2}.
\end{equation}
Below, we present two lemmas regarding $\tilde{\mathcal{D}}$-absorption and tail estimates. Let $\tilde{\mathcal{D}}$ denote the universe of closed tempered random sets in $X:=\ell^{2}$. A set $\mathcal{D}\in\tilde{\mathcal{D}}$ if $\PP$ as:
\begin{equation}\label{huanzengji}
	\lim_{t\to+\infty} {\rm{e}}^{-\rho t} \sup_{\Phi^{a}\in\mathcal{D}(\theta_{-t}\omega)}\|\Phi^{a}\|_{X} = 0, \ \forall \rho\textgreater0.
\end{equation}
This polynomial decay ensures the asymptotic absorption property required for attractor construction.
\begin{lemma}\label{quanjusuijixishou}
	For any $\omega \in \Omega$, $\tau \in \mathbb{R}$ and $E = \{E(\tau, \omega) : \tau \in \mathbb{R}, \omega \in \Omega\} \in \mathscr{D}$, there exists a $T := T(\tau, \omega, E,  a) > 0$ such that for all $t \geq T$ and $\Phi_0 \in \mathscr{D}(\theta_{-t}\omega)$, the following holds:
	\begin{equation}
	\|\Phi^a(t, \theta_{-t}\omega, \Phi_0)\|_{\mathbb{X}}^2 \leq R(a, \omega),
	\end{equation}
	where
	\begin{equation}\label{suijixishoubanjin}
	R(a, \omega) =\eta + 2c_{3} \|g\|_{\frac{p+2}{p+1}}^{\frac{p+2}{p+1}} \int_{-\infty}^{0} {\rm{e}}^{-az(\theta_{s}\omega)-\int_{0}^{s}2az(\theta_{h}\omega)dh - (8\lambda-2\gamma)s} ds,
	\end{equation}
	where $c_{3}$ will be defined in the following proof.
\end{lemma}
\begin{proof}[proof]
For convenience, we substitute $U^{a}$ with $U$ in the proof of this lemma. Taking the inner product of \eqref{suijigl2} with $U$, then we take the real part
\begin{equation}\label{suijilianxuweibuguji1}
	\frac{1}{2}\frac{d}{dt}\|U\|^{2} = \lambda\|D^{+}U\|^{2} -\gamma\|U\|^{2} -k\|U\|^{p+2}_{p+2}{\rm{e}}^{paz(\theta_{t}\omega)} + az(\theta_{t}\omega)\|U\|^{2} + \operatorname{Re}({\rm{e}}^{-az(\theta_{t}\omega)}g,U).
\end{equation}
Using a similar method in Theorem \ref{dingli2.1}, we have
\begin{equation*}
\begin{split}
	\lambda\|D^{+}U\|^{2} &\leq 4\lambda\|U\|^{2},\\
	\operatorname{Re}({\rm{e}}^{-az(\theta_{t}\omega)}g,U) &\leq k\|U\|^{p+2}_{p+2}{\rm{e}}^{paz(\theta_{t}\omega)} + \frac{p+1}{p+2}\frac{{\rm{e}}^{-az(\theta_{t}\omega)}}{[k(p+2)]^{\frac{1}{p+1}}}\|g\|_{\frac{p+2}{p+1}}^{\frac{p+2}{p+1}},
\end{split}
\end{equation*}
for convenience, we denote
\begin{equation*}
\begin{split}
	c_{3}=\frac{p+1}{p+2}\frac{1}{[k(p+2)]^{\frac{1}{p+1}}},
\end{split}
\end{equation*}
then we can rewrtie the \eqref{suijilianxuweibuguji1}
\begin{equation}\label{suijilianxuweibuguji2}
	\frac{d}{dt}\|U\|^{2} \leq 2(4\lambda-\gamma+az(\theta_{t}\omega))\|U\|^{2} + 2c_{3}\|g\|_{\frac{p+2}{p+1}}^{\frac{p+2}{p+1}}{\rm{e}}^{-az(\theta_{t}\omega)}.
\end{equation}
By the Gronwall's inequality on the $[\tau-t,r]$ with $r\geq\tau-t$, and substituting $\omega$ by $\theta_{-\tau}\omega$, it follows that
\begin{equation}\label{suijilianxuweibuguji3}
\begin{split}
	\|U_{t}\|^{2} &\leq {\rm{e}}^{\int_{-t}^{r-\tau} 2az(\theta_{s}\omega)ds + (8\lambda-2\gamma)(r-\tau+t)}\| U_{\tau-t} \|^{2} + 2c_{3}\int_{-t}^{r-\tau} {\rm{e}}^{-az(\theta_{s}\omega)} {\rm{e}}^{\int_{s}^{r-\tau} 2az(\theta_{h}\omega dh)} {\rm{e}}^{(8\lambda-2\gamma)(r-\tau-s)} \|g\|_{\frac{p+2}{p+1}}^{\frac{p+2}{p+1}} ds\\
	& = {\rm{e}}^{\int_{-t}^{r-\tau} 2az(\theta_{s}\omega)ds + (8\lambda-2\gamma)(r-\tau+t)}\| U_{\tau-t} \|^{2} \\
	&\quad+ 2c_{3}{\rm{e}}^{(8\lambda-2\gamma)(r-\tau)} {\rm{e}}^{\int_{r-\tau}^{0}2az(\theta_{h}\omega)dh} \int_{-t}^{r-\tau} {\rm{e}}^{-az(\theta_{s}\omega)-\int_{0}^{s}2az(\theta_{h}\omega)dh - (8\lambda-2\gamma)s} \|g\|_{\frac{p+2}{p+1}}^{\frac{p+2}{p+1}} ds.
\end{split}
\end{equation}
According to \eqref{OUxingzhi}, there exists a $T_{1}=T_{1}(\tau,\omega,E,a_{0})(0\textless a\textless a_{0})$ such that for any $t\geq T_{1} \textgreater0$, we have
\begin{equation*}
	|z(\theta_{-t}\omega)| \leq \frac{(-8\lambda+4\gamma)t}{4a_{0}},\quad |\int_{0}^{-t}z(\theta_{r}\omega)dr|\leq \frac{(-8\lambda-4\gamma)t}{8a_{0}}.
\end{equation*}
Thus, for any $t\geq T_{1}\textgreater0$, we get
\begin{equation*}
	c_{3}\int_{-t}^{-T_{1}} {\rm{e}}^{-az(\theta_{s}\omega)-\int_{0}^{s}2az(\theta_{h}\omega)dh - (8\lambda-2\gamma)s} \|g\|_{\frac{p+2}{p+1}}^{\frac{p+2}{p+1}} ds \leq c_{3}\int_{-t}^{-T_{1}} {\rm{e}}^{(\gamma-4\lambda)s}\|g\|_{\frac{p+2}{p+1}}^{\frac{p+2}{p+1}}ds.
\end{equation*}
Since $g\in\ell^{2}$, we have
\begin{equation}\label{suijilianxuweibugujibudengshi1}
	c_{3}\int_{-t}^{-T_{1}} {\rm{e}}^{(\gamma-4\lambda)s} \|g\|_{\frac{p+2}{p+1}}^{\frac{p+2}{p+1}}ds \textless +\infty.
\end{equation}
Thus, we get
\begin{equation*}
	c_{3}\int_{-t}^{r-\tau} {\rm{e}}^{-az(\theta_{s}\omega)-\int_{0}^{s}2az(\theta_{h}\omega)dh - (8\lambda-2\gamma)s} \|g\|_{\frac{p+2}{p+1}}^{\frac{p+2}{p+1}} ds \leq c_{3}\int_{-\infty}^{r-\tau} {\rm{e}}^{-az(\theta_{s}\omega)-\int_{0}^{s}2az(\theta_{h}\omega)dh - (8\lambda-2\gamma)s} \|g\|_{\frac{p+2}{p+1}}^{\frac{p+2}{p+1}} ds,
\end{equation*}
where the integral is convergent due to \eqref{suijilianxuweibugujibudengshi1}. Therefore, we yield
\begin{equation*}
\begin{split}
	\|U(r,\tau-t,\theta_{-\tau}\omega,U_{\tau-t})\|^{2} &\leq   {\rm{e}}^{\int_{-t}^{r-\tau} 2az(\theta_{s}\omega)ds + (8\lambda-2\gamma)(r-\tau+t)}\| U_{\tau-t} \|^{2} \\
		&\quad+ 2c_{3}{\rm{e}}^{(8\lambda-2\gamma)(r-\tau)} {\rm{e}}^{\int_{r-\tau}^{0}2az(\theta_{h}\omega)dh} \int_{-\infty}^{r-\tau} {\rm{e}}^{-az(\theta_{s}\omega)-\int_{0}^{s}2az(\theta_{h}\omega)dh - (8\lambda-2\gamma)s} \|g\|_{\frac{p+2}{p+1}}^{\frac{p+2}{p+1}} ds.
\end{split}
\end{equation*}
When $r=\tau$, we have
\begin{equation*}
\begin{split}
	\|U(r,\tau-t,\theta_{-\tau}\omega,U_{\tau-t})\|^{2}&\leq {\rm{e}}^{\int_{-t}^{0} 2az(\theta_{s}\omega)ds + (8\lambda-2\gamma)t}\| U_{\tau-t} \|^{2}\\
	&\quad + 2c_{3} \int_{-\infty}^{0} {\rm{e}}^{-az(\theta_{s}\omega)-\int_{0}^{s}2az(\theta_{h}\omega)dh - (8\lambda-2\gamma)s} \|g\|_{\frac{p+2}{p+1}}^{\frac{p+2}{p+1}} ds,
\end{split}
\end{equation*}
also from the \eqref{OUxingzhi}, we have
\begin{equation*}
	\int_{-\infty}^{0} {\rm{e}}^{-az(\theta_{s}\omega)-\int_{0}^{s}2az(\theta_{h}\omega)dh - (8\lambda-2\gamma)s} ds \textless +\infty.
\end{equation*} 
Note that $\left\{E(\tau,\omega):\tau\RR,\omega\in\Omega\right\}\in\mathcal{D}$ is tempered. Therefore, for any $U_{\tau-t}\in E(\tau-t,\theta_{-t}\omega)$, we derive
\begin{equation*}
	\lim_{t\to+\infty} {\rm{e}}^{\int_{-t}^{0} 2az(\theta_{s}\omega)ds + (8\lambda-2\gamma)t}\| U_{\tau-t} \|^{2} = 0.
\end{equation*}
This completes the proof.
\end{proof}

According to Lemma \ref{quanjusuijixishou}, the random dynamical system $\Gamma^{a}(\cdot,\omega)$ has a random $\tilde{\mathcal{D}}$-absorbing set $\mathcal{G}\in\tilde{\mathcal{D}}$, which is described by 
\begin{equation}\label{suijixishouji}
	\mathcal{G}^{a} = \left\{\Phi^{a}\in X: \|\Phi^{a}\|_{X}^{2}\leq R(a,\omega)\right\},\ \omega\in\Omega.
\end{equation}
\begin{lemma}\label{suijilianxuweibuguji}
For any $\delta > 0, \mathscr{D} \in \tilde{\mathscr{D}}$ and $\omega \in \Omega$, there exist $N(\delta, \omega)$, $T(\delta, \omega, \mathscr{G}^{a}(\theta_{-t}\omega))$, $\Phi_0 \in \mathscr{D}(\theta_{-t}\omega)$ such that
\begin{equation}\label{suijiweibugujijieguo}
	\bigl\| U(t, \theta_{-t}\omega, U_0) \bigr\|_{X(|j| \geq N(\delta, \omega))}^2 \leq \delta, \quad \forall \ t \geq T(\delta, \omega, \mathscr{G}^{a}(\theta_{-t}\omega)).
\end{equation}
\end{lemma}
\begin{proof}[proof]
Taking inner product of \eqref{suijigl2} with $\xi_{l}U$ and then we take the real part, where $\xi_{l}$ is defined in Theorem \ref{dingli3.1}, from \eqref{diyixiang1} and \eqref{diyixiang2}, we get
\begin{equation}\label{suijilianxuweibuguji11}
\begin{split}
	\frac{1}{2} \frac{d}{dt} \sum_{j \in \mathbb{Z}} \xi_{l,j}|U_{j}|^{2} &\leq 4\lambda\sum_{j \in \mathbb{Z}}\xi_{l,j}|U_{j}|^{2} + (4\lambda+2\mu)\frac{c_{2}}{l}\|U\|^{2} - \gamma\sum_{j \in \mathbb{Z}}\xi_{l,j}|U_{j}|^{2} - k{\rm{e}}^{paz(\theta_{t}\omega)}\sum_{j \in \mathbb{Z}} \xi_{l,j}|U_{j}|^{p+2}\\
	&\quad + {\rm{e}}^{-az(\theta_{t}\omega)}\operatorname{Re}(g,\xi_{l}U) + az(\theta_{t}\omega)\sum_{j \in \mathbb{Z}}|U_{j}|^{2}\\ 
	&\leq (4\lambda-\gamma+az(\theta_{t}\omega))\sum_{j \in \mathbb{Z}}|U_{j}|^{2} + c_{3}{\rm{e}}^{-az(\theta_{t}\omega)}\sum_{j \in \mathbb{Z}}\xi_{l,j}|g_{j}|^{\frac{p+1}{p+2}} + (4\lambda+2\mu)\frac{c_{2}}{l}\|U\|^{2}.
\end{split}
\end{equation}
Using Gronwall's inequality on \eqref{suijilianxuweibuguji11} from $T_{k} = T_{k}(\omega)\geq0$ to $t\geq T_{k}$, and then substituting $\omega$ with $\theta_{-t}\omega$, we have
\begin{equation}\label{suijilianxuweibuguji12}
\begin{split}
	&\quad\sum_{j \in \mathbb{Z}}\xi_{l,j}|U_{j}(t,\theta_{-t}\omega,U_{0}(\theta_{-t}\omega))|^{2}\\
	 &\leq {\rm{e}}^{-2(\gamma-4\lambda)(t-T_{k}) + \int_{T_{k}}^{t} 2az(\theta_{s-t}\omega)ds}\sum_{j \in \mathbb{Z}}\xi_{l,j}|U_{j}(T_{k},\theta_{-t}\omega,U_{0}(\theta_{-t}\omega))|^{2}\\
	& \quad+ (4\lambda+2\mu) \frac{c_{2}}{l}\int_{T_{k}}^{t} {\rm{e}}^{-2(\gamma-4\lambda)(t-\tau)+\int_{\tau}^{t}2az(\theta_{s-t}\omega)ds} \|U(\tau,\theta_{-t}\omega,U_{0}(\theta_{-t}\omega))\|^{2}d\tau\\
	&\quad+ c_{3}\int_{T_{k}}^{t} {\rm{e}}^{-az(\theta_{\tau-t}\omega)}{\rm{e}}^{-2(\gamma-4\lambda)(t-\tau)+\int_{\tau}^{t}2az(\theta_{s-t}\omega)ds}d\tau \sum_{j \in \mathbb{Z}}\xi_{l,j}|g_{j}|^{\frac{p+2}{p+1}}\\
	&=: J_{1} + J_{2} + J_{3}.
\end{split}
\end{equation}
According to \eqref{huanzengji}, $J_{1}\to0$ as $t\to\infty$. This implies that for any $\delta\textgreater0$, there exists a $T_{1}=T_{1}(\delta,\omega,\mathcal{G}^{a}(\theta_{-t}\omega))\geq T_{k}$ such that
\begin{equation}\label{suijilianxuweibugujidiyixiang}
	J_{1}(t) \leq \frac{\delta}{3} {\rm{e}}^{-2z(\omega)},\ t\geq T_{1}.
\end{equation}
For Lemma \ref{quanjusuijixishou}, we can find $T_{2} = T_{2}(\delta,\omega,\mathcal{G}^{a}(\theta_{-t}\omega))\textgreater T_{k}(\omega)$ and $N_{1} = N_{1}(\delta,\omega)\textgreater0$, such that if $t\textgreater T_{2}$ and $l\textgreater N_{1}$, there holds
\begin{equation}\label{suijilianxuweibugujidierxiang}
	J_{2}(t) \leq \frac{\delta}{3} {\rm{e}}^{-2z(\omega)}.
\end{equation}
In face, by \eqref{OUxingzhi}, we have the following estimation:
\begin{equation*}
	\int_{T_{k}}^{t} {\rm{e}}^{-az(\theta_{\tau-t}\omega)}{\rm{e}}^{-2(\gamma-4\lambda)(t-\tau)+\int_{\tau}^{t}2az(\theta_{s-t}\omega)ds}d\tau \leq \int_{0}^{t} {\rm{e}}^{-az(\theta_{\tau-t}\omega)}{\rm{e}}^{-2(\gamma-4\lambda)(t-\tau)+\int_{\tau}^{t}2az(\theta_{s-t}\omega)ds}d\tau \textless \infty,
\end{equation*}
since $g\in\ell^{2}$, there exists $N_{2}= N_{2}(\delta,\omega)\in\NN$ such that
\begin{equation}\label{suijilianxuweibugujidisanxiang}
	J_{3}(t) \leq \frac{\delta}{3} {\rm{e}}^{-2z(\omega)},\ \forall l \textgreater N_{2}.
\end{equation}
Therefore, by substituting \eqref{suijilianxuweibugujidiyixiang}-\eqref{suijilianxuweibugujidisanxiang} into \eqref{suijilianxuweibuguji12}, we obtain
\begin{equation}\label{suijilianxuweibugujijieguo}
	\sum_{|j|\geq N(\delta,\omega)}\xi_{l,j}|U_{j}(t,\theta_{-t}\omega,U_{0}(\theta_{-t}\omega))|^{2} \leq \delta {\rm{e}}^{-2z(\omega)},\ t\geq\max\left\{T_{1},T_{2}\right\},\ l\geq\max\left\{N_{1}, N_{2}\right\},
\end{equation}
provided $N(\delta,\omega)$ is large enough. Based on \eqref{suijilianxuweibugujijieguo}, we can deduce that
\begin{equation*}
	\sum_{|j|\geq N(\delta,\omega)}\xi_{l,j}|u_{j}(t,\theta_{-t}\omega,u_{0}(\theta_{-t}\omega))|^{2} \leq \delta.
\end{equation*}
Thus, the proof is complete.
\end{proof}

Recall that a random compact set $\mathscr{A}(\omega)$ of $X$ is called a random $\tilde{\mathscr{D}}$-attractor for a random dynamical system $\Phi$ with semigroup $\mathscr{H}(\cdot)$ if $\mathscr{A} := \{\mathscr{A}(\omega)\} \in \tilde{\mathscr{D}}$, and $\mathscr{A}$ satisfies the invariance property, meaning that for any $t \geq 0$ and $\omega \in \Omega$, $\Phi(t, \omega)\mathscr{A}(\omega) = \mathscr{A}(\theta_t\omega)$. Additionally, $\mathscr{A}$ is $\tilde{\mathscr{D}}$-attracting, i.e.,
\begin{equation}
	\lim_{t \to \infty} \mathrm{dist}_X \bigl( \Phi(t, \theta_{-t}\omega) \mathscr{D}(\theta_{-t}\omega), \mathscr{A}(\omega) \bigr) = 0, \ \forall \mathscr{D} \in \tilde{\mathscr{D}}, \ \omega \in \Omega.
\end{equation}
The subsequent theorem establishes the existence of the random attractor, as well as the upper semi-convergence relationship between the random attractor and the global attractor.

\begin{theorem}\label{dingli5.1}
The random dynamical system $\Phi^a$ generated by the problem \eqref{suijigl2} has a random attractor $\mathscr{A}(\omega)$ in $X$. Furthermore, $\mathscr{A}(\omega)$ upper semi-converges to the global attractor $\mathscr{A}$, which means that
\begin{equation}\label{5.21}
\lim_{a \to 0} \mathrm{dist}_X \bigl( \mathscr{A}^{a}(\omega), \mathscr{A} \bigr) = 0, \ \mathbb{P}\text{-a.s.} \ \omega \in \Omega.
\end{equation}
\end{theorem}
\begin{proof}[proof]
By Lemma \ref{suijilianxuweibuguji}, the equation \eqref{suijigl} possesses a unique global random $\tilde{\mathscr{D}}$ attractor given by
\begin{equation*}
	\mathscr{A}(\omega) = \bigcap_{\tau \geq T_k} \overline{\bigcup_{t \geq \tau} \Phi(t, \theta_{-t}\omega, \mathscr{G}(\theta_{-t}\omega))} \in X, \ \tau \in \mathbb{R}, \ \omega \in \Omega.
\end{equation*}
Since the radius $R(a, \omega)$ in \eqref{suijixishoubanjin} is bounded as $a \to 0$, and the tail estimate in Lemma \ref{suijilianxuweibuguji} is uniform for all $a \in (0, a^*]$, uniform asymptotic compactness can be ensured. Moreover, by employing a method similar to that in Lemma \ref{yinli5.4}, which will be presented later, we can establish the convergence of the random system $\Gamma^{a}(\cdot, \omega)$ to the semigroup $\mathscr{H}(\cdot)$ as $a \to 0$. In summary, the upper semi-convergence \eqref{5.21} can be derived by applying the abstract result from \cite{35}.
\end{proof}
At the last of this subsection, we study the semi-continuity of random attractor as following:
\begin{theorem}\label{suijishangbanlianxu}
	The random attractor is upper semi-continuous, that is, for selected $t$, we have
	\begin{equation*}
		\lim_{a\to a_{0}} d_{\ell^{2}}(\mathcal{A}^{a}(\omega),\mathcal{A}^{a_{0}}(\omega)) =0 , \ \forall \omega\in\Omega\ \text{and}\ a_{0} \neq 0.
	\end{equation*}
\end{theorem}
\begin{proof}[Proof]
As for this Theorem, proving the convergence as below is sufficient:
\begin{equation}\label{shangbanlianxu1}
	\lim_{n\to \infty} \| \Phi^{a_{n}}(t,\theta_{-t},\omega,\Phi^{a_{n}}_{0}) - \Phi^{a_{0}}(t,\theta_{-t},\omega,\Phi^{a_{0}}_{0})  \| = 0,
\end{equation}
where $a_{n}\to a_{0}\neq0$, and $\Phi_{0}^{a_{n}}\to\Phi_{0}^{a_{0}}$. Let $u^{n}(t,\omega) = U^{a_{n}}(t,\omega,u_{0}) - U^{a_{0}}(t,\omega,u_{0})$, by submiting into \eqref{suijigl2}, we have
\begin{equation}\label{shangbanlianxu2}
\begin{split}
	\frac{du^n}{dt} =& (\lambda+i\mu)\Lambda u^{n}-(\gamma-i\beta)u^{n}-(k+i\nu)\left[ |U^{a_{n}}|^{p}U^{a_{n}}{\rm{e}}^{pa_{n}z(\theta_{t}\omega)}-|U^{a_{0}}|^{p}U^{a_{0}}{\rm{e}}^{pa_{0}z(\theta_{t}\omega)} \right] \\
	&+\quad ({\rm{e}}^{-a_{n}z(\theta_{t}\omega)}-{\rm{e}}^{-a_{0}z(\theta_{t}\omega)})g + a_{n}z(\theta_{t}\omega)U^{a_{n}} - a_{0}z(\theta_{t}\omega)U^{a_{0}}.
\end{split}
\end{equation}
Obviously, we konw that $U^{a_{0}} = {\rm{e}}^{(a_{n}-a_{0})z(\theta_{t},\omega)}$, then we take the real part of the inner product of the \eqref{shangbanlianxu2} with $u^{n}$, for convinence we use $z$ instead of $z(\theta_{t}\omega)$ in the rest of this proof.
\begin{equation}\label{shangbanlianxu3}
\begin{split}
	\frac{1}{2}\frac{d}{dt}\|u^{n}\|^{2} &= \lambda\|D u^{n}\|^{2} - \gamma \|u^{n}\|^{2} - \operatorname{Re}\left\{(k+i\nu)\left[ |U^{a_{n}}|^{p}U^{a_{n}} \left({\rm{e}}^{a_{n}pz} -{\rm{e}}^{a_{n}(p+1)z-a_{0}z}  \right)  \right]  \right\} \\
	&\quad+  \operatorname{Re}\left\{(1-{\rm{e}}^{(a_{n}-a_{0})z}) ({\rm{e}}^{-a_{n}z}-{\rm{e}}^{-a_{0}z})(g,U^{a_{n}})  \right\} + \operatorname{Re}\left\{ \left( z(a_{n} - a_{0}{\rm{e}}^{a_{n}-a_{0}z})U^{a_{n}}, (1-{\rm{e}}^{(a_{n}-a_{0})z})U^{a_{n}}    \right)  \right\}.
\end{split}
\end{equation}
form \cite{Liu-2024a}, we can obtain
\begin{equation}\label{shangbanlianxu4}
\begin{split}
	&\sup_{t\in[0,T]} |{\rm{e}}^{pa_{n}z} - {\rm{e}}^{pa_{0}z}| \to 0,\ \text{as}\ n\to\infty,\\
	&\sup_{t\in[0,T]} |{\rm{e}}^{a_{n}z} - {\rm{e}}^{a_{0}z}| \to 0,\ \text{as}\ n\to\infty.
\end{split}
\end{equation}
Then we use a method which will be used in Lemma \ref{yinli5.4}, we can get the \eqref{shangbanlianxu1} and complete the proof.
\end{proof}
\subsection{The semi-continuity of truncated random attractor}
Using the notations for the two matrices $\Lambda_{m}$ and $D_{m}^{+}$ denoted in previous section, we consider the truncation
of the random lattice system \eqref{suijigl2} in the space $X_m$ as follows:
\begin{equation}\label{suijijieduangl}
	\frac{dU^{a,m}}{dt} = (\lambda+i\mu)\Lambda_{m} U^{a,m} - (\gamma+i\beta)U^{a,m} - (k+i\nu)(|U^{a,m}|^{p})U^{a,m} {\rm{e}}^{paz(\theta_{t}\omega)} + {\rm{e}}^{-az(\theta_{t}\omega)}g^{m} + aU^{a,m}z(\theta_{t}\omega),
\end{equation}
with the initial condition $U^{a,m}(0) = U_{0}^{a}\in X_{m}$. For this truncated random system \eqref{suijijieduangl}, there exists a unique solution. This solution generates a continuous random dynamical system $\Xi_{m}^{a}:\RR^{+}\times \Omega \times X_{m}\to X_{m}$ corresponding to $a\in(0,a^{*}],\ m\in\NN$, which is denoted as 
\begin{equation*}
	\Xi_{m}^{a}(t,\omega)U_{0}^{a} = \Phi^{a,m}(t,\omega,U_{0}^{a}),\ \forall t\geq0,\omega\in\Omega,U_{0}^{a}\in X_{m}.
\end{equation*}
Let $\widetilde{\mathscr{D}}_m$ be the restriction of $\widetilde{\mathscr{D}}$ on $X_m$, that is, $\mathscr{D} \in \widetilde{\mathscr{D}}_m$ if and only if
\[
\lim_{t \to +\infty} e^{-\rho t} \sup\bigl\{\|\Phi^\varepsilon\|_{X_m} : \Phi^\varepsilon \in \mathscr{D}(\theta_{-t} \omega)\bigr\} = 0, \ \omega \in \Omega, \ \text{for all } \rho > 0.
\]

To derive the upper semi-convergence between the truncated semigroup and the global attractor, we proceed to demonstrate the following $\widetilde{\mathscr{D}}_m$-absorbing set and tail estimates for the truncated random system. These correspond to Lemmas \ref{quanjusuijixishou} and \ref{suijilianxuweibuguji}.
\begin{lemma}\label{yinli5.3}
	Let $a \in (0, a^*]$, and $m \in \mathbb{N}$. Then the random dynamical system $\Xi_m^\varepsilon$ has a random $\widetilde{\mathscr{D}}_m$-absorbing set given by
	\begin{equation}\label{suijijieduanxishouji}
		\mathscr{G}_m^a(\omega) = \left\{ \Phi_m^a \in X_m : \| \Phi \|_{X_m}^2 \leq R(\varepsilon, \omega) \right\}, \ \omega \in \Omega,
	\end{equation}
	where $R(\varepsilon, \omega)$ is given by \eqref{suijixishoubanjin} and is independent of $m$. Moreover, for each $\delta > 0$, $\mathscr{D}_m \in \widetilde{\mathscr{D}}_m$ and $\omega \in \Omega$, there exist $T(\delta) > 0$ and $I(\delta) \in \mathbb{N}$ such that for all $t \geq T$, $m > I(\delta)$ and $\Phi_0^m \in \mathscr{D}_m(\theta_{-t}\omega)$,
	\begin{equation*}
		\left| \Phi^{a, m}(t, \theta_{-t}\omega, \Phi_0^m) \right| \leq 	\sum_{1 \leq |j| \leq m} \left| u_j^{\varepsilon, m} \right|^2 \leq \delta.
	\end{equation*}
\end{lemma}
\begin{theorem}\label{dingli5.2}
Assume $a \in (0, a^*]$. For each $m \in \mathbb{N}$, the random dynamical system $\Xi_m^a$ has a random $\widetilde{\mathscr{D}}$-attractor $\mathscr{A}_m^\varepsilon(\omega) \subset X_m$ with the following convergence:
\begin{equation}\label{suijijieduandaoquanju}
	\lim_{m \to \infty} \mathrm{dist}_X\left( \mathscr{A}_m^a(\omega), \mathscr{A}^a(\omega) \right) = 0, \quad \forall \omega \in \Omega,
\end{equation}
where $\mathscr{A}_m^a(\omega)$ is naturally embedded in $X$, and $\mathscr{A}^a(\omega)$ is the random attractor given in Theorem \ref{dingli5.1}.
\end{theorem}
\begin{proof}[Proof]
	By Lemma \ref{suijilianxuweibuguji}, the semi-dynamical system $\Phi^{a}$ is asymptotically compact in $\widetilde{\mathscr{D}}_m$, which implies the existence of the random attractor. Similar to the case of stochastic PDE on expanding domains as given in \cite{Li-2019,Li-2019b}, we can obtain the convergence in \eqref{suijijieduandaoquanju}. In addition, using a method similar to that in Lemma \ref{yinli5.4} later, we can prove the convergence $\Gamma_m^a \to \Gamma^a$ as $m \to \infty$. Moreover, we can verify that the absorption and tail estimate in Lemma \ref{yinli5.3} are uniform for sufficiently large $m$, and thus we can prove that the asymptotic compactness of $\Gamma_m^a(t, \omega)$ is also uniform if $m$ is sufficiently large.
\end{proof}

The deterministic system \eqref{gl1} on the finite dimensional space $X_{m}$ can be written as 
\begin{equation}\label{youxiangl}
	\frac{du^{m}}{dt} = (\lambda+i\mu)\Lambda_{m} u^{m} - (\gamma-i\beta) u^{m} - (k+i\nu) |u^{m}|^{p} u^{m} + g^{m}.
\end{equation}
For each $m \in \mathbb{N}$, there exists a unique solution of \eqref{youxiangl}, which generates a semigroup $\mathscr{H}_m : \mathbb{R}^+ \times X_m \to X_m$ defined by
\begin{equation}\label{jieduanbanqun}
	\mathscr{H}_m(t)u_0^m = u^m(t, u_0^m), \ \forall t \geq 0, \ u_0^m \in X_m.
\end{equation}
To prove that the truncated semigroup $\mathscr{H}_m(\cdot)$ has a global attractor $\mathscr{A}_m$ in $X_m$ and that $\mathscr{A}_m$ converges to $\mathscr{A}_m^a(\omega)$ as $a \to 0$. We need to show that, for any $m \in \mathbb{N}$, the semigroup $\mathscr{H}_m(\cdot)$ is the limit of the truncated random dynamical system $\Xi_m^a(\cdot, \omega)$ as $a \to 0$.
\begin{lemma}\label{yinli5.4}
For any $m\in\NN$, when $|\Phi_{0}^{a}-u_{0}^{a}|_{X_{m}}\to0$ as $a\to0$, the solutions of \eqref{suijigl2} and \eqref{youxiangl} satisfy the following:
\begin{equation}\label{suijiyouxianbijin}
	\lim_{a\to0} |\Phi^{a,m}(t,\theta_{-t},\omega,\Phi_{0}^{a})-u^{a}(t,u_{0}^{m})|_{X_{m}} = 0,\ \forall t\textgreater 0,\omega\in\Omega.
\end{equation}
\end{lemma}
\begin{proof}[Proof]
For convenience, we use $u,U^{a}$ and $g$ instead of $u^{m},U^{a,m}$ and $g^{m}$. Let $u^{a}(t,\omega) = U^{a}(t,\omega,u_{0}^{m}) - u(t,u_{0}),\ \forall t\geq 0$. Then, by subtracting \eqref{youxiangl} from \eqref{suijigl2}, we obtain
\begin{equation}\label{suijiyouxiancha}
	\frac{du^a}{dt} = (\lambda+i\mu)\Lambda_{m}u^{a}-(\gamma-i\beta)u^{a}-(k+i\nu)\left[ |U^{a}|^{p}U^{a}{\rm{e}}^{paz(\theta_{t}\omega)}-|u|^{p}u \right] + ({\rm{e}}^{-az(\theta_{t}\omega)}-1)g + az(\theta_{t}\omega)U^{a}.
\end{equation}
Then we take the inner product of \eqref{suijiyouxianbijin} with $u^{a}$ in $X_{m}$, by taking the real part, we have
\begin{equation}\label{suijiyouxianbijin2}
\begin{split}
	\frac{1}{2}\frac{d}{dt}\|u^{a}\|^{2} &= \lambda\|D^{+}_{m}u^{a}\|^{2} - \gamma \|u^{a}\|^{2} - \operatorname{Re}\left\{(k+i\nu)\left[ |U^{a}|^{p}U^{a}{\rm{e}}^{paz(\theta_{t}\omega)}-|u|^{p}u \right]\overline{u^a}\right\} \\
	&\quad+ \operatorname{Re}({\rm{e}}^{-az(\theta_{t}\omega)}-1)(g,u^{a}) + \operatorname{Re}(az(\theta_{t}\omega)U^{a},u^{a}).
\end{split}
\end{equation}
Since $u = {\rm{e}}^{az(\theta_{t}\omega)} U^{a}$, that is $u^{a} = (1-{\rm{e}}^{az(\theta_{t}\omega)})U^{a}$, then we have
\begin{equation*}
\begin{split}
	\operatorname{Re}\left\{(k+i\nu)\left[ |U^{a}|^{p}U^{a}{\rm{e}}^{paz(\theta_{t}\omega)}-|u|^{p}u \right]\overline{u^a}\right\} &= \operatorname{Re}\left\{ (k+i\nu)(1-{\rm{e}}^{az(\theta_{t}\omega)})^{2}|U^{a}|^{p}U^{a}{\rm{e}}^{paz(\theta_{t}\omega)}\overline{U^{a}} \right\}\\
	&=k{\rm{e}}^{paz(\theta_{t}\omega)}(1-{\rm{e}}^{az(\theta_{t}\omega)})^{2}|U^{a}|^{p+2},\\
	\operatorname{Re}({\rm{e}}^{-az(\theta_{t}\omega)}-1)(g,u^{a}) &\leq ({\rm{e}}^{-az(\theta_{t}\omega)}-1)(1-{\rm{e}}^{az(\theta_{t}\omega)})|(g,U^{a})|\\
	& \leq k{\rm{e}}^{paz(\theta_{t}\omega)}(1-{\rm{e}}^{az(\theta_{t}\omega)})^{2}|U^{a}|^{p+2} + c_{3}({\rm{e}}^{-az(\theta_{t}\omega)}-1)^{2} \|g\|_{\frac{p+2}{p+1}}^{\frac{p+2}{p+1}},\\
	\operatorname{Re}(az(\theta_{t}\omega)U^{a},u^{a}) &\leq \frac{1}{2}\left( \|U^{a}\|^{2} + \|u^{a}\|^{2}  \right).
\end{split}
\end{equation*}
Let $\tau\in\RR,\omega\in\Omega,T\textgreater0$, and $\zeta\in[0,1)$. Since $\omega$ is continuous on $\RR$, there exists $a_{1}=a_{1}(\omega,T,\zeta)\textgreater0$ such that for each $a\in(0,\min\left\{a_{1},a^{*}\right\})$ and $t\in[\tau,\tau+T]$, we have(see \cite{Su-2025,Wang-2024})
\begin{equation}\label{zxingzhi}
	|{\rm{e}}^{-az(\theta_{t}\omega)}-1| \leq \zeta.
\end{equation}
Thus, we get
\begin{equation}\label{suijiyouxianbijin3}
	\frac{d}{dt}\|u^{a}\|^{2} \leq (8\lambda-2\gamma+az(\theta_{t}\omega))\|u^{a}\|^{2} + 2c_{3}\zeta^{2}\|g\|_{\frac{p+2}{p+1}}^{\frac{p+2}{p+1}} + az(\theta_{t}\omega)\|U^{a}\|^{2}.
\end{equation}
By using Gronwall's inequality to \eqref{suijiyouxianbijin3} on $[0,t]$, with $\omega$ replaced by $\theta_{-t}\omega$, we obtain
\begin{equation}\label{suijiyouxianbijin4}
\begin{split}
	\|u^{a}(t,\theta_{-t}\omega)\|^{2} &\leq {\rm{e}}^{(8\lambda-2\gamma)t+\int_{-t}^{0}az(\theta_{s}\omega)ds}\|u_{0}^{a}\|^{2} + (2c_{3}\zeta^{2}\|g\|_{\frac{p+2}{p+1}}^{\frac{p+2}{p+1}} + az(\theta_{t}\omega)\|U^{a}\|^{2})\int_{-t}^{0} {\rm{e}}^{(8\lambda-2\gamma)s + \int_{s}^{0}az(\theta_{r}\omega)dr} ds\\
	&:= Z_{1} + Z_{2}.
\end{split}
\end{equation}
Since we have supposed $|\Phi_{0}^{a} - u_{0}^{m}|_{X_{m}}\to0$ as $a\to0$ and combined with \eqref{OUxingzhi}, for any $t\textgreater0$, we get
\begin{equation*}
	Z_{1}={\rm{e}}^{(8\lambda-2\gamma)t+\int_{-t}^{0}az(\theta_{s}\omega)ds}\|u_{0}^{a}\|^{2} \leq {\rm{e}}^{(8\lambda-2\gamma)t+at\sup_{s\in[-t,0]}z(\theta_{s}\omega)ds}\|u_{0}^{a}\|^{2}\to0,\ \text{as},\ a \to 0.
\end{equation*}
In fact, from \eqref{zxingzhi} we konw that $\zeta\to0$ as $a\to0$. On the other hand, according to Lemma \ref{quanjusuijixishou}, we know that $U^{a}$ is bounded, with $g\in\ell^{2}$, it is obvious that $Z_{2}\to0$ as $a\to0$. Thus we get
\begin{equation*}
	\lim_{a\to0} |\Phi^{a,m}(t,\theta_{-t},\omega,\Phi_{0}^{a})-u^{a}(t,u_{0}^{m})|_{X_{m}} = 0,\ \forall t\textgreater 0,\omega\in\Omega.
\end{equation*}
This completes the proof.
\end{proof}
\begin{theorem}\label{dingli5.3}
Assume $a \in (0, a^*]$. The truncated semigroup $\mathscr{H}_m(\cdot)$, generated by \eqref{suijiyouxianbijin}, has a global attractor $\mathscr{A}_m$ in $X_m$, which converges to the attractor $\mathscr{A}_m^a(\omega)$ as $a \to 0$, that is
\begin{equation}\label{suijiyouxianshoulian}
\lim_{a \to 0} \mathrm{dist}_{X_m} \bigl( \mathscr{A}_m^a(\omega), \mathscr{A}_m \bigr) = 0, \ \forall \omega \in \Omega.
\end{equation}
\end{theorem}

\begin{proof}[Proof]
Note that since $|g^m| \leq \|g\|,\ g\in\ell^{2}$, and by applying the methodology developed in Lemma \ref{yinli2.1}, we can establish that the truncated semigroup $\mathscr{H}_m(\cdot)$ possesses a bounded absorbing set, which is explicitly given by
\[
\mathscr{G}_m := \left\{ u^m \in X_m : \|u^m\|_{X_m} \leq r^* := \sqrt{\eta + \frac{c_{1}}{\gamma - 4\lambda}} \right\} = \mathscr{H}_{X_m}(0, r^*).
\]
Since $\mathscr{G}_m$ is compact in $X_m$, it follows that $\mathscr{H}_m(\cdot)$ has a global attractor $\mathscr{A}_m$ in $X_m$.

To prove the result in \eqref{suijiyouxianshoulian}, it suffices to verify the three essential conditions of the abstract theorem(see \cite{Bates-2009,Li-2015,35}).

(i) By Lemma \ref{yinli5.4}, as $|U_0^a - u_0^m|_{X_m} \to 0$ when $a \to 0$, for each $t > 0$ and $\omega \in \Omega$, it follows that
\[
\lim_{a \to 0} \left| \Xi_m^a \bigl(t, \theta_{-t}\omega, U_0^a\bigr) - \mathscr{H}_m(t)u_0^m \right|_{X_m} = 0.
\]

(ii) By the Lebesgue dominated convergence theorem, one can deduce that
\[
\begin{aligned}
\lim_{a \to 0} R(a, \omega) &= \eta + 2c_{3} \|g\|_{\frac{p+2}{p+1}}^{\frac{p+2}{p+1}} \lim_{a\to0}\int_{-\infty}^{0} {\rm{e}}^{-az(\theta_{s}\omega)-\int_{0}^{s}2az(\theta_{h}\omega)dh - (8\lambda-2\gamma)s} ds \\
&= \eta + \frac{c_{1}}{\gamma-4\lambda} \leq r^* \,^2, \quad \text{as } t \to \infty.
\end{aligned}
\]

(iii) $\mathscr{A}_m^a(\omega) \subseteq \mathscr{G}_m^a(\omega)$ holds for all $a \in (0, a^*]$, it follows that for all $\omega \in \Omega$, the absorbing radius $R(a, \omega)$ is increasing with respect to $a \in (0, a^*]$, we have
\[
\bigcup_{0 < a \leq a^*} \mathscr{A}_m^a(\omega) \subseteq \bigcup_{0 < a \leq a^*} \mathscr{G}_m^a(\omega) \subseteq \mathscr{G}_m^{a*}(\omega)
\]

Therefore, we conclude that the union $\bigl\{ \mathscr{A}_m^a(\omega) : 0 < a \leq a^* \bigr\}$ forms a precompact set in the finite dimensional space $X_m$. This establishes the third condition and the proof is complete.
\end{proof}
Finally, we give the theorem of the upper semi-continuity of truncated random attractor.
\begin{theorem}\label{jieduanshangbanlianxu}
	The truncated random attractor is upper semi-continuous, for fixed $t$,
	\begin{equation}
		\lim_{n\to \infty} dist_{X_{m}}(\mathcal{A}^{a_{n}}_{m}(\omega),\mathcal{A}^{a_{0}}_{m}) = 0, \ \forall\omega\in\Omega\ \text{and} \ a_{0}\neq0.
	\end{equation} 
\end{theorem}
\begin{proof}[Proof]
 This Theorem can be proved by the same method used in Theorem \ref{suijishangbanlianxu} and lemma \ref{yinli5.4}.  
\end{proof}

\subsection*{Conflict of interest}
The authors have no conflicts to disclose.

\subsection*{Availability of date and materials}
Not applicable.

\end{document}